\documentclass[reqno,oneside,12pt]{amsart}

\usepackage[T1]{fontenc}
\usepackage{times,mathptm}
\usepackage{amssymb,epsfig,verbatim,xcolor, graphicx, enumerate, tikzsymbols, transparent}
\usepackage{stmaryrd, mathtools}

\theoremstyle{plain}

\newtheorem{thm}{Theorem}[section]
\newtheorem{cor}[thm]{Corollary}
\newtheorem{pro}[thm]{Proposition}
\newtheorem{lem}[thm]{Lemma}
\newtheorem{proposition-principale}[thm]{Proposition principale}
\newtheorem{thm-principal}{Main Theorem}

\theoremstyle{definition}
\newtheorem{que}{Question}[section]
\newtheorem{eg}[thm]{Example}
\newtheorem{rem}[thm]{Remark}

\newenvironment{defi-G}
{\noindent{\bf Definition.}\it}{\\}

\newenvironment{thm-M}
{\noindent{\bf Main Theorem.}\it }{}

\newenvironment{thm-A}
{\noindent{\bf Theorem A.--}\it }{}

\newenvironment{thm-AA}
{\noindent{\bf Theorem A'.}\it}{\\ }

\newenvironment{thm-B}
{\noindent{\bf Theorem B.--}\it}{\\ }

\newenvironment{thm-C}
{\noindent{\bf Theorem C.--}\it}{\\ }

\newenvironment{thm-D}
{\noindent{\bf Theorem D.--}\it}{\\ }

\newenvironment{thm-E}
{\noindent{\bf Theorem E.--}\it}{\\ }

\newenvironment{thm-BB}
{\noindent{\bf Theorem B'.}\it}

\def\C{\mathbf{C}}
\def\R{\mathbf{R}}
\def\Q{\mathbf{Q}}

\def\Z{\mathbf{Z}}
\def\N{\mathbf{N}}
\def\bfK{\mathbf{K}}
\def\bfk{\mathbf{k}}
\def\VR{{{\mathfrak{o}}_{\mathbf{K}}}}
\def\bfKbar{{\overline{\mathbf{K}}}}
\def\mK{{\mathfrak{m}}_\bfK}
\def\bfL{\mathbf{L}}
\def\VRL{{{\mathfrak{o}}_{\mathbf{L}}}}
\def\bfF{\mathbf{F}}

\DeclarePairedDelimiter\av{\vert}{\vert}
\DeclarePairedDelimiter\set{{\{}}{{\}}}

\def\vx{{{\mathbf{x}}}}
\def\vy{{{\mathbf{y}}}}
\def\vz{{{\mathbf{z}}}}
\def\vt{{{\mathbf{t}}}}
\def\vs{{{\mathbf{s}}}}
\def\vb{{{\mathbf{b}}}}

\def\P{\mathbb{P}}

\def\A{\mathbb{A}}
\def\Hyp{\mathbb{H}}

\def\U{{\mathcal{U}}}

\def\Aut{{\sf{Aut}}}
\def\Bir{{\sf{Bir}}}
\def\Diff{{\sf{Diff}}}
\def\Bij{{\sf{Bij}}}

\def\Per{{\mathrm{Per}}}
\def\Ker{{\mathrm{Ker}}}

\def\Lie{{\sf{Lie}}}
\def\red{{\mathrm{red}}}
\def\modulo{{\mathrm{mod}}\; } 
\def\vect{{\mathrm{vect}}}

\def\Isom{{\sf{Isom}}}

\def\NS{{\mathrm{NS}}}

\def\Supp{{\rm{Supp}}}

\def\PGL{{\sf{PGL}}\,}

\def\GL{{\sf{GL}}\,}

\def\SL{{\sf{SL}}\,}

\def\End{{\sf{End}}\,}

\def\Tr{{\sf{Tr}}}

\def\det{{\sf{det}}}

\newcommand{\Id}{{\rm id}}
\newcommand{\Idop}{{\rm Id}}

\def\Ind{{\text{Ind}}}
\def\Exc{{\text{Exc}}}

\def\Gm{{\mathbb{G}}_m}
\def\Ga{{\mathbb{G}}_a}

\def\dist{{\sf{dist}}}

\newcommand{\seunguk}[1]{{\color{green!80!black}*}\marginpar{\tiny  \color{green!80!black} SJ: #1}}
\newcommand{\serge}[1]{{\color{red}*}\marginpar{\tiny  \color{red} SC: #1}}

\addtocounter{section}{0}             
\numberwithin{equation}{section}       

\begin{document}

\setlength{\baselineskip}{0.56cm}        
%
%
\title[Orbits of automorphism groups of affine surfaces over $p$-adic fields]
{Orbits of automorphism groups of affine surfaces over $p$-adic fields}
\date{2023/2024}
\author{Serge Cantat, Seung uk Jang}
\address{CNRS, IRMAR - UMR 6625 \\ 
Universit{\'e} de Rennes 
\\ France}
\email{serge.cantat@univ-rennes.fr}
 \email{seung-uk.jang@univ-rennes1.fr}
%
%

%
%

%
%

\begin{abstract}
We study orbit closures and stationary measures for groups of automorphisms of $p$-adic affine surfaces. 
 \end{abstract}

\maketitle

\setcounter{tocdepth}{1}

\tableofcontents

\vfill

{\small{The research activities of the authors  are partially funded by the European Research Council (ERC GOAT 101053021).  }}

\pagebreak

\section{Introduction} 

\subsection{Affine varieties over local fields} If $V\subset \A^N$ is an affine variety defined over a ring $R$, we denote by $\Aut(V_R)$ the group of automorphisms of $V$ defined over $R$. Specifically, $f\colon V\to V$ is an element of $\Aut(V_R)$ if it is an automorphism and both $f$ and $f^{-1}$ are defined by formulas with coefficients in $R$. 
If $R$ is the valuation ring of a local field (see  Section~\ref{par:local_fields}), we shall endow $\A^N(R)= R^N$ with the product topology, and then $V(R)$ with the induced topology, which is independent of the embedding of $V$ into $\A^N$. 
The group $\Aut(V_R)$ determines a subgroup of ${\mathsf{Homeo}}(V(R))$, the group of homeomorphisms of~$V(R)$.

\subsection{Stationary measures} Let $T$ be a topological space. Let $\mu$ be a probability measure on ${\mathsf{Homeo}}(T)$. A stationary measure, with respect to $\mu$, is a probability measure $\nu$ on $T$ such that  $\int \nu(f^{-1}(B))\; d\mu(f)=\nu(B)$ for every Borel subset $B\subset T$. Let $\Gamma$ be a subgroup of ${\mathsf{Homeo}}(T)$. By definition, a stationary measure for $\Gamma$ is a stationary measure with respect to a probability measure $\mu$ on ${\mathsf{Homeo}}(T)$ whose support generates $\Gamma$.  Given a compact $\Gamma$-invariant set $K\subset T$, and a probability measure $\mu$ on $\Gamma$, there is always at least one $\mu$-stationary measure $\nu$ whose support is contained in $K$. 

\subsection{Orbit closures} An automorphism of an affine surface $X$ is said to be loxodromic when its (first) dynamical degree is greater than $1$: we refer to Section~\ref{par:automorphisms} for a description of the different types of automorphisms.
Our goal is to prove the following theorem. 
\smallskip

\begin{thm-A} 
Let $X$ be a normal affine surface defined over $\Z_p$, the ring of $p$-adic integers. 
Let $\Gamma$ be a subgroup of $\Aut(X_{\Z_p})$ containing 
a loxodromic element and a non loxodromic element of infinite order. Then, 
\begin{enumerate}[\rm(1)] 
\item $X(\Z_p)$ contains at most finitely many finite $\Gamma$-orbits;
\item if a  $\Gamma$-orbit is infinite, its closure is a clopen subset of $X(\Z_p)$;
\item the decomposition of the compact set $X(\Z_p)$ into closures of $\Gamma$-orbits is a partition into (at most) countably many clopen subsets and (at most) finitely many finite orbits;
\item this partition is finite if and only if  every finite orbit of $\Gamma$ in $X(\Z_p)$ is an isolated  subset of $X(\Z_p)$. 
\end{enumerate}
Moreover, if $\mu$ is a probability measure, whose support 
generates $\Gamma$, then each orbit closure $O$ is the support of a $\Gamma$-invariant probability measure, and this invariant measure is the unique stationary measure supported in $O$. 
\end{thm-A}

\smallskip

Parts of Theorem~A hold when $\Z_p$ is replaced by the valuation ring of an arbitrary $p$-adic local field, see in particular Theorem~D in Section~\ref{par:proof_of_thmA_part1}.
The important example of Markov surfaces is  studied in Section~\ref{par:markov_surfaces} (see Theorem~E in Section~\ref{par:markov_general_parameter}). 

Theorem~A is motivated   by the results of William Goldman concerning the ergodic theory of the mapping class group action on character varieties~\cite{Goldman:Annals},   those of Jean Bourgain, Alex Gamburd, and Peter Sarnak concerning Markov surfaces over finite fields~\cite{Bourgain-Gamburd-Sarnak:Markov}, and the classification of stationary measures recently obtained by the first author with Christophe Dupont and Florestan Martin-Baillon for Markov surfaces~\cite{cantat-dupont-martinbaillon}. 
 
\subsection*{Acknowledgements} We would like to thank Jason Bell, Christophe Dupont, Sébastien
Gouëzel, Vincent Guirardel, Florestan Martin-Baillon, and Aftab Patel for interesting discussions. We also thank 
Alireza Salehi Golsefidy for sharing his ongoing work with Nattalie Tamam on the dynamics of the mapping class group action on character varieties.

\section{From automorphisms to flows} 
 
In this section, we recall an important theorem proven by Jason Bell; the version that we shall use is due to Bjorn Poonen (see~\cite{Bell, Bell:corrigendum, Poonen2014}). Then, following Sections 2 and 9 of~\cite{CantatXie2018}, we explain how this theorem can be applied to study the dynamics of groups of automorphisms (the reader should also consult~\cite[\S{3.1}]{Abboud:IMRN}, \cite[Prop. 22--23]{Amerik:survey-padic}, and \cite[Thm. 3.3]{bell-ghioca-tucker}). In Subsections~\ref{par:lie_dimension} and~\ref{par:general_finiteness_result}, we prove a general finiteness result on orbit closures (see Theorem~B). In the rest of the paper, we will see how this general result may be used to reach our Theorem~A.

\subsection{Local fields of characteristic $0$}\label{par:local_fields}
A non-archimedian local field $\bfK$ of characteristic zero, which will be idiomatically called a \emph{$p$-adic local field}, is the same thing as a
finite extension of $\Q_p$ for some prime $p$, with the topology induced by the (unique) extension of the $p$-adic absolute 
value $\av{\cdot}_p$. We denote the absolute value of $\bfK$ by $\av{\cdot}$; the normalization is $\av{p}=1/p$. The valuation ring is $\VR = \set{x\in \bfK\; ; \; \av{x}\leq 1}$; its group of units is $\set{x\in \bfK\; ;\; \av{x}=1}$.
The valuation group is the image $\av{\bfK^\times}\subset \R_+^\times$ of the absolute value, and it is cyclic, equal to $\av{\pi}^\Z$ for some element $\pi\in \bfK^\times$ with $\av{\pi}<1$. Such a $\pi$ is called a uniformizer, and it is unique up to multiplication by a unit of $\VR$. The ring $\VR$ contains a unique maximal ideal, namely $\mK=\set{x\in \VR\; ; \; \av{x}<1 }=\pi\VR$. The residue field is $\bfk=\VR/\mK$, and it is a finite extension of $\bfF_p$.

We consider $\VR$ as the closed unit disk of $\bfK$ and we refer to  $\VR^m$ as the (closed) polydisk of dimension $m$. 

\subsection{Analytic diffeomorphisms of the polydisk} \label{par:analytic_diffeomorphism_polydisk}

Set $\U=\VR^m$. 
An analytic function $\varphi$ on the polydisk $\U$ is an analytic function in the sense of Tate, that 
is, if we denote by $\vx=(\vx_1, \ldots, \vx_m)$ the coordinates, and by $\vx^I=\vx_1^{i_1}\cdots \vx_m^{i_m}$ the monomial associated with a multi-index $I=(i_1, \ldots, i_m)$, then $\varphi$ is defined by a power series 
\begin{equation}
\varphi(\vx)=\sum_I a_I \vx^I
\end{equation}
with coefficients $a_I\in \bfK$ such that $\av{a_I}$ goes to $0$ as $I$ goes to infinity (i.e. the length $\av{I}=i_1+\cdots +i_m$ goes to $+\infty$).  The set of analytic functions with coefficients in a subring $R\subset \bfK$ will be denoted by $R\langle \vx\rangle$. We will be particularly interested in the cases $R=\VR$ and $R=\Z_p\subset \VR$. Similarly, an analytic map $\U\to \bfK^n$ is a map defined by $n$ analytic functions $(\varphi_1,\ldots, \varphi_n)$. This notion is usually called Tate analyticity (\footnote{Note that it requires a global power series expansion on the polydisk $\U$. As such, it is not a direct analogue of the classical definition of complex analytic functions, which requires only  a local analytic expansion around each point.}). 

An analytic endomorphism of $\U$ is just an analytic map $g\colon \U\to \U$ that can be written as
$g(\vx)=(g_1(\vx), \ldots, g_m(\vx))
$ for some analytic functions $g_i\in \VR\langle \vx\rangle$. Endomorphisms form an $\VR$-module $\End^{an}(\U)$, and a monoid for the composition. Invertible elements (with respect to the composition) form a group, the group $\Diff^{an}(\U)$ of analytic diffeomorphisms of the polydisk. One can also define analytic vectors fields and analytic flows. By definition, a flow parametrized by $\VR$ is an analytic map $\Phi\colon \VR\times \U \to \U$, $(\vt,\vx)\mapsto \Phi^\vt(\vx)$, such that 
$\Phi^{\vs+\vt}(\vx)=\Phi^{\vs}(\Phi^{\vt}(\vx)).$ The vector field associated with such a flow is $\Theta(\vx)=\partial_{\vt}(\Phi^{\vt}(\vx))_{\vt=0}$. An integral curve of this vector field is the same as an orbit $t\in \VR \mapsto \Phi^t(z)\in \VR^m$, for fixed $z\in \U$. 

Let $V\subset \A^N$ be an affine variety of dimension $m$. Let $U$ be a subset of $V$. If there is an analytic map $\varphi\colon \U\to \A^N(\bfK)$ such that $\varphi(\U)=U$ and $\varphi$ is a diffeomorphism from $\U$ to $U$, we say that $U$ is a polydisk of $V$  (see~\cite[Prop. 3.3, 3.4]{CantatXie2018}). 

\subsection{Reduction and the Bell-Poonen theorem} \label{par:reduction_and_bell_poonen}
If $\ell$ is a positive integer and $\varphi=\sum_I a_I\vx^I$ is an element of $\VR\langle \vx\rangle$, its reduction modulo 
$\mK^\ell$ is the formal power series $\overline{\varphi}$ with coefficients in $\VR/\mK^\ell$ obtained by reduction of the coefficients $a_I$ modulo $\mK^\ell$. Thus, $\overline{\varphi}=\sum_I \overline{a_I} \vx^I$. Since $\av{a_I}$ goes to $0$ as $\av{I}$ goes to infinity, this power series is a polynomial in $(\vx_1, \ldots, \vx_m)$ with coefficients in the ring $\VR/\mK^\ell$. 
Recall that $\mK=\pi\VR$ for some uniformizer $\pi$. Let $e$ be the integer such that $\av{\pi}^e=\av{p}$. 
One writes $\varphi\equiv \psi \;(\modulo p^c)$ if there is an integer $\ell\geq ec$ such that $\varphi\equiv \psi \;(\modulo \mK^{\ell})$.

\begin{thm}[Bell-Poonen]\label{thm:bell_poonen}
Let $f$ be an analytic endomorphism of the polydisk $\U=\VR^m$ with $f\equiv \Id \;(\modulo p^c)$ for some real number $c>1/(p-1)$.
Then, $f$ is an analytic diffeomorphism of $\U$ and there exists a unique analytic flow $\Phi\colon \VR\times \U\to \U$ such that $\Phi^n(\vx)=f^n(\vx)$ for every $n\in \Z$; moreover
$\Phi^t(\vx)\equiv\vx\pmod{p^c}$
for all $t\in \VR$.
\end{thm}

Thus, if the congruence $f\equiv \Id\; (\modulo p^{c})$ holds, the dynamics of $f$ on $\U$ are ``embedded'' into the dynamics of a flow $\Phi$. For instance,  {\emph{if $\bfK=\Q_p$, then $\VR=\Z_p$ and $\Z$ is dense in $\Z_p$, so that the closure of an $f$-orbit $f^\Z(x)$ for the $p$-adic topology equals the trajectory $\Phi^{\Z_p}(x)$}}. As a continuous image of a compact set $\Z_p\times\set{x}$, the trajectory $\Phi^{\Z_p}(x)$ is closed in $\U$. 

We provide a proof of the Bell-Poonen theorem in the appendix.

\begin{rem} \label{rem:ceci-nest-pas-une-norme}
As the residue field $\bfk$ of $\VR$ is finite, the following property is \emph{not} sufficient for the congruence $f\equiv\Id\pmod{\mK^\ell}$: for all $x\in\U$, we have $f(x)-x\in(\mK^\ell)^m$. 
For instance, let $q=|\bfk|$ be the residue field order, and set $f(\vx_1)=\vx_1+(\vx_1-\vx_1^q)^\ell$ a univariable polynomial. Then we have $f\not\equiv\Id\pmod{\mK^\ell}$ yet $f(x_1)\equiv x_1\pmod{\mK^\ell}$ for all $x_1\in\VR$. This issue does not arise if $\bfK$ is extended to its algebraic closure; one can also avoid it by a finite extension of $\bfK$ whose residue field is sufficiently large (relative to $f$ and $\ell$).
\end{rem}

\subsection{Automorphisms}
Let $V$ be an affine variety, defined over $\VR$, where $\VR$ is the valuation ring of $\bfK$, as described above.

\subsubsection{Reduction} Reduction $\modulo \mK^\ell$ provides a homomorphism 
\begin{equation}
f\in \Aut(V_{\VR})\mapsto \overline{f} \in \Aut(V_{\VR/\mK^\ell}).
\end{equation} 

\begin{eg}Consider the affine plane $\A^2$ and the group $\Aut(\A^2_\VR)$. 
This group contains all affine transformations $f(\vx_1,\vx_2)=L(\vx_1,\vx_2)+T$ with $L\in \GL_2(\VR)$ and $T\in \A^2(\VR)$, as well as all 
elementary automorphisms $g(\vx_1,\vx_2)=(a\vx_1+b(\vx_2),c\vx_2+d)$, where $a, c, d\in \VR$, $b\in \VR[\vx_2]$, and $\av{a}=\av{c}=1$.
Any affine or elementary automorphism of $\A^2_\bfk$ is the reduction modulo $\mK$ of such an affine or elementary automorphism
of $\A^2_\VR$. By the theorem of Jung and van der Kulk (see~\cite{Lamy:Jung}, Thm. 2, and references therein), the reduction homomorphism $\Aut(\A^2_\VR)\to \Aut(\A^2_\bfk)$ 
is onto. \end{eg}

\begin{eg} Consider a Markoff cubic surface $M$, defined by the equation
\begin{equation}
\vx_1^2+\vx_2^2+\vx_3^2+\vx_1\vx_2\vx_3 = A\vx_1+B\vx_2+C\vx_3+D
\end{equation}
for some parameters $A$, $B$, $C$, $D$ in $\VR$.
A theorem of El'Huti shows that the  homomorphism $\Aut(M_\VR)\to \Aut(M_\bfk)$ 
is onto (see~\cite[Thm. 2]{el-huti}). \end{eg}

Let us embed $V_\VR$ into an affine space $\A^N_\VR$.  
As explained in the introduction, we endow $\A^N(\bfK)= \bfK^N$ with the product topology, and $V(\bfK)$ (resp. $V(\VR)$) with the induced topology. 
This topology does not depend on the choice of an embedding into an affine space. With this topology, 
$V(\VR)$ is compact.

For any integer $\ell\geq 1$,  denote by 
$\red_\ell\colon V(\VR)\to V(\VR/\mK^\ell)$ the reduction map modulo $\mK^\ell$. The group $\Aut(V_\VR)$ acts by permutation on $V({\VR/\mK^\ell})$, which yields a homomorphism $\Aut(V_\VR)\to  \Bij(V(\VR/\mK^{\ell}))$. 
Fix a point $z\in V(\VR/\mK^\ell)$. Its preimage $\red_\ell^{-1}(z)$ in $V(\VR)$ is an open and closed subset of $V(\VR)$. By the Hensel lemma, if $\ell$ is large enough and
if $z$ is a smooth point of $V(\bfk)$, the clopen set $\red_\ell^{-1}(z)$ is a polydisk $\simeq \VR^m$; for details, see \cite[Prop. 2.2]{bell-ghioca-tucker} and \cite[Prop. 3.4]{CantatXie2018}. Then, the stabilizer of $z\in V({\VR/\mK^\ell})$ in $\Aut(V_\VR)$ coincides with the stablizer of $\red_\ell^{-1}(z)$ and is isomorphic to a subgroup of  $\Diff^{an}(\VR^m)$. 

\subsubsection{Good polydisks (see~\cite[\S{3.2}]{CantatXie2018})} Suppose, for simplicity, that $V$ is smooth. Fix an integer $\ell_1\geq 1$ such that each preimage $\red_{\ell_1}^{-1}(z)$, for $z\in V(\VR/\mK^{\ell_1})$, is diffeomorphic to a polydisk. By the finiteness of $\VR/\mK^{\ell_1}$, this provides a finite cover of $V(\VR)$ by polydisks $\U_i$, with $i$ between $1$ and $N_1:=\vert V(\VR/\mK^{\ell_1}) \vert$. In what follows, we identify each $\U_i$ with the polydisk $\VR^m$ via some analytic diffeomorphism.

Let $\Gamma$ be a subgroup of $\Aut(V_\VR)$. 
Each of these polydisks is invariant under the action of the kernel $\Gamma_1$ of the representation $\Gamma\to \Bij(V(\VR/\mK^{\ell_1}))$ and, by restriction to each $\U_i$, $\Gamma_1$ determines $N_1$ subgroups 
\begin{equation}
\Gamma_{1}(i) \subset \Diff^{an}(\VR^m), 
\end{equation} 
where $\VR^m \simeq \U_i$. Now, choose $\ell_2$ such that $p^4\in \mK^{\ell_2}$. Each $\Gamma_{1}(i)$ 
acts, after reduction modulo $\mK^{\ell_2}$, by polynomial automorphisms on $(\VR/\mK^{\ell_2})^m$. Denote by 
$\Gamma_{0}(i)\subset \Gamma_1$ the kernel of the action on the tangent space of $(\VR/\mK^{\ell_2})^m$; by this we mean that in $\U_i$, every element $f$ of $\Gamma_{0}(i)$ satisfies 
\begin{equation} 
f(x)=x \; (\modulo \mK^{\ell_2}) \; \text{ and } \; Df_x=\Id \; (\modulo  \mK^{\ell_2})
\end{equation}  
for every $x\in (\VR/\mK^{\ell_2})^m$. The intersection of all $\Gamma_0(i)$'s is a normal subgroup $\Gamma_0$ of $\Gamma$.  

Now, pick an element $f$ in $\Gamma_0$ and a point $x$ in $V(\VR)$. Then $x$ is in some $\U_i\simeq \VR^m$. In this polydisk $\VR^m$, if we conjugate $f$ by the translation $t_x$ that maps the origin $o$ to $x$, we can write $f$ as a series 
\begin{equation}
f(\vx)=A_0+A_1(\vx)+A_2(\vx)+\cdots +A_n(\vx)+\cdots
\end{equation}
where 
\begin{enumerate}
\item each $A_n\colon \bfK^m\to \bfK^m$ is a homogeneous polynomial mapping of degree $n$ with coefficients in $\VR$;
\item $A_0=0 \;(\modulo p^4)$ and $A_1(\vx)=\vx \;(\modulo p^4)$.
\end{enumerate}
If we conjugate $f$ by the homothety $h_p\colon\vx\mapsto p\vx$, then $h_p^{-2}\circ f \circ h_p^2$ satisfies the hypothesis of Theorem~\ref{thm:bell_poonen}. Since the origin corresponds to $x$ before conjugating by $t_x$, and $h_p^2$ maps $\VR^m$ to $(p^2\VR)^m$, we see that $x$  is contained in a (smaller) polydisk $\U(x)\subset \U_i$ that is invariant by $\Gamma_0$ and where each element of $\Gamma_0$ satisfies the hypotheses of  Theorem~\ref{thm:bell_poonen}.


\begin{thm}\label{thm:good_finite_cover} Let $\bfK$ be a non-archimedian local field of characteristic $0$ and $\VR$ be its valuation ring.
Let $V$ be a smooth affine variety defined over $\VR$. Let $\Gamma$ be a subgroup of $\Aut(V_\VR)$.
There is a finite index subgroup $\Gamma_0$ of $\Gamma$ with the following property. Given any $x\in V(\VR)$, there is 
a compact open neighborhood $\U(x)$ of $x$ such that
\begin{enumerate}[\rm(1)]
\item $\U(x)$ is diffeomorphic to the polydisk $\VR^m$;
\item $\U(x)$ is $\Gamma_0$-invariant;
\item for any $f\in \Gamma_0$, there is a unique analytic flow $\Phi_f \colon \VR\times \U(x)\to \U(x)$ such that $f^n_{\vert \U(x)}=\Phi_f^n$ for every $n\in \Z$.
\end{enumerate}
\end{thm}

We shall say that the orbit $\set{\Phi^t(x)\; ;\; t\in \VR}$ is the {\bf{analytic trajectory}} of $x$ determined by $f$.
\begin{rem}
When $V$ is singular, the same result holds locally on the complement of the singularities.
If the singularities of $V$ are quotient singularities, one can get a finite cover as in Theorem~\ref{thm:good_finite_cover}, except that some of the $\U(x)$ will be the quotient $\VR^m/G_x$ of the polydisk by a finite group. Then the third assertion is better stated on $\VR^m$ rather than that on the quotient. 
\end{rem}

\begin{rem}
If $V$ is defined over $\VR$ and $W\subset V(\bfK)$ is a $\Gamma$-invariant compact open subset, then the same statement holds on $W$. 
\end{rem}

\begin{rem}\label{rem:no_torsion}
 If $f$ is an element of $\Gamma_0$, $y$ is in $\U(x)$ and  $f^k(y) = y$ for some $k\geq 1$, then 
$\Phi_f^{kn}(y)=y$ and, by the principle of isolated zeroes (see~\cite[Sec. VI.2]{robert:book}), $\Phi_f^t(y)=y$ for every $t\in \VR$; taking $t=1$ shows that $f(y)=y$. 
Thus, in $V(\VR)$, every periodic point of $f$ is fixed. In particular, $\Gamma_0$ is torsion free
as soon as $V(\VR)$ is Zariski dense in $V$. See~\cite{bass-lubotzky} for a more general result.
\end{rem}

\subsection{The distribution $L_\Gamma(x)$ and its dimension $s_\Gamma(x)$}\label{par:lie_dimension}

Let $x$ be a smooth point of $V(\VR)$,  $\U\subset V(\VR)$ be a polydisk containing $x$, and  $f$ be an element of $\Gamma$ stabilizing $\U$. If $f=\Phi_f^1$ for some flow $\Phi_f\colon \VR\times \U\to \U$, we define 
\begin{equation}
\label{eqn:champ-vectoriel-definition}
\Theta_f(x)=\left(\frac{\partial \Phi_f^t(x)}{\partial t}\right)_{\vert t=0}.
\end{equation}

Changing $f$ to $f^k$ for some $k\in\Z\setminus\set{0}$ changes $\Theta_f(x)$ to $k\Theta_f(x)$. Thus, in the tangent space $T_xV$, viewed as a vector space over  $\Q_p$, we have  $\Q_p\Theta_f(x)=\Q_p\Theta_{f^k}(x)$ for every $k\neq 0$. This space is a line if and only if $\Theta_f(x)\neq 0$, if and only if $f(x)\neq x$ (see Remark~\ref{rem:no_torsion}). We denote this line by 
\begin{equation}
[\Theta_f(x)]\in \P(T_xV)
\end{equation}
where $\P(T_xV)=(T_xV\setminus \set{0})/\Q_p^\times$ (note that $T_xV$ is viewed as a vector space over $\Q_p$, rather than over $\bfK$). 
Now, suppose $\U$ is stabilized by a finite index subgroup $\Gamma_0\subset \Gamma$ such that every element of $\Gamma_0$ is in a flow. The subspace 
\begin{equation}
L_{\Gamma_0}(x)=\vect_{\Q_p}(\Theta_f(x); f\in \Gamma_0)\subset T_xV
\end{equation}
is equal to $L_{\Gamma_1}(x)$ if $\Gamma_1$ has finite index in $\Gamma_0$. Thus, it depends only on $\Gamma$ and we denote it by $L_{\Gamma}(x)$. By definition, $L_{\Gamma}(x)$ is the {\bf{tangent space to the action}} of $\Gamma$ at $x$. Its dimension, as a $\Q_p$-vector space, is denoted by $s_\Gamma(x)$. 

\begin{pro}
Let $\U\subset V(\VR)$ be a $\Gamma_0$-invariant polydisk, for some finite index subgproup $\Gamma_0$ of $\Gamma$. The function $s_\Gamma \colon \U\to \Z_+$ is semi-continuous in the analytic topology: this means that the sets $\set{x\in \U\; ;\; s_\Gamma(x)\leq k}$ are $\Q_p$-analytic subsets (in the sense of Tate). 
\end{pro}

Note that in this statement the polydisk $\U\simeq \VR^m$ is seen as a polydisk $\Z_p^{km}$ of dimension $km$ over $\Z_p$, with $\VR\simeq \Z_p^k$.  

\begin{proof}
Let $f_1$, $\ldots$, $f_{k+1}$ be elements of $\Gamma_0$. Then, the set $\set{x\in \U\; ;\; s_\Gamma(x)\leq k}$ is contained in the $\Q_p$-analytic set $\set{x\in \U\; ; \; \Theta_{f_1}(x)\wedge \cdots \wedge \Theta_{f_{k+1}}(x)=0}$, 
and is equal to the intersection of these sets as we vary the $f_i$.
\end{proof}

\subsection{A general finiteness result}\label{par:general_finiteness_result}$\; $ 

\smallskip

\begin{thm-B} Let $V$ be a smooth affine variety defined over $\VR$. Let $\Gamma$ be a subgroup of $\Aut(V_\VR)$. Let $B\subset V(\VR)$ be a $\Gamma$-invariant compact subset.
If $s_\Gamma(x)=[\bfK:\Q_p]\dim(V)$ for every $x\in B$, 
then the orbit closures of $\Gamma$ in $B$ form a partition of $B$ into finitely many compact open subsets.
\end{thm-B}

\begin{proof}
Let $x$ be a point of $V(\VR)$. Let $\U\subset V(\VR)$ be a polydisk containing $x$ which is stabilized by a finite index subgroup $\Gamma_0$ of $\Gamma$, as in Theorem~\ref{thm:good_finite_cover}. 

Set $N= [\bfK:\Q_p]\dim(V)$. If $s_\Gamma(x)=N$, the orbit closure $C_0(x):={\overline{\Gamma_0(x)}}$ contains an open neighborhood of $x$. Indeed, we can choose elements $f_1$, $\ldots$, $f_N$ in $\Gamma_0$ such that $(\Theta_{f_i}(x))$, $1\leq i\leq N$, form a basis of $T_xV$ as a $\Q_p$-vector space. Then, if we denote by $\Phi_{i}$ the flow associated with $f_i$, the map 
\begin{equation}
(t_1, \ldots, t_N)\in \Z_p^N\mapsto \Phi_{f_N}^{t_N}\circ \cdots\circ \Phi_{f_1}^{t_1}(x)\in V(\VR)
\end{equation}
 is a local diffeomorphism.
On the other hand, given any $N$-tuple $(t_1, \ldots, t_N) $ and any $\epsilon >0$, there are integers $(n_1, \ldots, n_N)$ such that $\av{t_i-n_i}_p\leq \epsilon$ for each $i$. Thus,   there is a neighborhood of $x$ in which $\Gamma_0(x)$ is $\epsilon$-dense. Since this holds for every $\epsilon>0$, $C_0(x)$ contains a neighborhood of $x$. 

Set $C(x)={\overline{\Gamma(x)}}$. Since $\Gamma_0$ has finite index in $\Gamma$, $C(x)$ is made of finitely many images $g(C_0(x))$, with $g\in \Gamma/\Gamma_0$. In particular, $C(x)$ is compact and open. As $C(x)$ is open, the sets $C(x)$, for $x\in V(\VR)$, are disjoint: if two orbit closures intersect, they coincide. This proves that the orbit closures form a partition of $V(\VR)$ into clopen subsets, and since $V(\VR)$ is compact, this partition contains only finitely many atoms. 
\end{proof}

\begin{rem}\label{rem:orbit_closure}  The same argument shows the following. 
{\emph{
If $s_\Gamma(x)=[\bfK:\Q_p]\dim(V)$ for every $x$ in the complement of some closed analytic subset $Z\subset V$, then 
the orbit closures form a partition of $V(\VR)\setminus Z(\VR)$ into (at most) countably many compact open subsets; this partition is locally finite. 
}}
\end{rem}

\section{Automorphisms of affine surfaces} \label{par:automorphisms}
In this section and the next one, we collect a few facts concerning automorphisms of affine surfaces: the leitmotif is to explain how general results concerning birational transformations of surfaces specialize --- and become more precise --- for automorphisms of affine surfaces. See, in particular, Theorem~C in Section~\ref{par:jonquieres_gmxgm}.
Here, we denote by $X_0$ an affine surface and by $X$ a projective completion of $X_0$. 
This notation is specific to this section (in the rest of the paper, $X$ is an affine surface). 

\subsection{Elliptic, parabolic, loxodromic} \label{sec:elliptic-parabolic-loxodromic}

Let $X$ be a projective surface, over an algebraically closed field $\bfK$. Let $\Bir(X)$ be its group of birational transformations. 
Let $H$ be a polarization of $X$; the degree of a rational map $f\colon X\dasharrow X$ with respect to $H$ is the intersection number 
\begin{equation}\label{eq:equivalence_of_degrees}
\deg_H(f)= (H\cdot f^*H).
\end{equation}
If $\varphi\colon Y\dasharrow X$ is a birational map and $H'$ is a polarization of $Y$, 
then 
\begin{equation}
c^{-1} \deg_H(f)\leq \deg_{H'}(\varphi^{-1} \circ f \circ \varphi)\leq c\deg_H(f)
\end{equation} 
for some constant $c\geq 1$ and all $f\in \Bir(X)$; in what follows, the degree is simply denoted by $\deg(\cdot)$.
Each element $f$ of $\Bir(X)$ satisfies exactly one of the following properties  (see~\cite[Thm. 4.6]{cantat:cremona_survey},~\cite[Thm. 4.6]{Cantat-Guiradel-Lonjou},~\cite[Sec. 3.2]{Cantat-Favre}). 
\begin{enumerate}[(a)]
\item[(a)] The sequence $(\deg(f^n))_{n\geq 0}$ is bounded, and then there is a smooth projective surface $Y$, a birational map $\varphi\colon Y\dasharrow X$, and a positive integer $k$ such that $f_Y:=\varphi^{-1}\circ f\circ \varphi$ is in $\Aut(Y)$ and $f_Y^k$ is in the connected component $\Aut(Y)^0$ of $\Id_Y$ in $\Aut(Y)$.
\item[(b)] The sequence $(\deg(f^n))$ grows linearly with $n\in \N$, and then there is a smooth projective surface $Y$, a birational map $\varphi\colon Y\dasharrow X$, a curve $B$, and a genus $0$ fibration $\pi\colon Y\to B$ such that $f_Y:=\varphi^{-1}\circ f\circ \varphi$ permutes the fibers of $\pi$. 
This fibration $\pi$ is uniquely determined by $f_Y$ (up to post-composition by an automorphism of $B$). Such a birational transformation is never conjugate to an automorphism of a {\emph{projective}} surface $Y$ by a birational map $Y\dasharrow X$.
\item[(b')] The sequence $(\deg(f^n))$ grows quadratically with $n\in \N$, and then there is a smooth projective surface $Y$, a birational map $\varphi\colon Y\dasharrow X$, a curve $B$, and a genus $1$ fibration $\pi\colon Y\to B$, such that $f_Y:=\varphi^{-1}\circ f\circ \varphi$ permutes the fibers of $\pi$. 
Moreover, on a relatively minimal model of the invariant fibration, $f$ becomes an automorphism.
\item[(c)] The sequence $(\deg(f^n))$ grows exponentially with $n\in \N$. In that case, $f$ does not preserve any pencil of curves. 
\end{enumerate}
In cases (b) and (b'), {\sl{genus $k$ fibration}} means that $\pi$ is a fibration with connected fibers and with smooth general fibers of genus $k$; and {\sl{$f_Y$ permutes the fibers of $\pi$}} means that there is an automorphism $f_B$ of $B$ such that $\pi \circ f_Y=f_B\circ \pi$.

In case (a), we say that $f$ is {\bf{elliptic}}; in cases (b) and (b') that $f$ is {\bf{parabolic}}; and in case (c) that $f$ is {\bf{loxodromic}}. Thus, there are two types of parabolic transformations: in case (b), $f$ is called a {\bf{Jonquières twist}}; in case (b') it is a {\bf{Halphen twist}}. 
In all cases, the limit 
\begin{equation}
\lambda_1(f)=\lim_{n\to +\infty} \deg(f^n)^{1/n}
\end{equation}  
exists and is, by definition, the dynamical degree of $f$; we have $\lambda_1(f)\geq 1$, and the inequality is strict if and only if $f$ is loxodromic. 

\begin{eg}
A non-loxodromic automorphism of the affine plane is conjugate to an elementary map $(\vx,\vy)\mapsto (a\vx+b, c\vy+P(\vx))$ and, as such, is elliptic. For any polynomial function $P(\vx)$ of degree $\geq 2$, the Hénon map 
$(\vx,\vy)\mapsto (\vy+P(\vx), \vx)$ is loxodromic, with dynamical degree $\deg(P)$.
\end{eg}

\begin{pro}\label{pro:not_conjugate_automorphism_projective} 
Let $X_0$ be an affine surface and $\Gamma$ be a subgroup of $\Aut(X_0)$. Assume there is a projective surface $Y$ and a birational map $\varphi_0\colon X_0\dasharrow Y$ such that the group $\Gamma_Y:=\varphi_0\Gamma\varphi_0^{-1}$ is contained in $\Aut(Y)$. Then, $\Gamma$ is contained in an algebraic subgroup of $\Aut(X_0)$; in particular, every element of $\Gamma$ is elliptic and $\Gamma$ is elementary.
\end{pro}
\begin{proof}
Let $\Exc(\varphi_0)$ be the union of the irreducible curves $E\subset X_0$ which are contracted by $\varphi_0$ (i.e.\ the strict transform of $E$ is a point). 
Let $E$ be such a curve and let $q\in Y$ be the point to which $E$ is contracted. 
If $f$ is in $\Gamma$ and $f_Y=\varphi_0\circ f\circ\varphi_0^{-1}$, then $f(E)$ is a curve contracted by $\varphi_0$ onto the point $f_Y(q)$. Thus, $\Gamma$ preserves $\Exc(\varphi_0)$. Similarly, the indeterminacy locus $\Ind(\varphi_0)$ is a finite $\Gamma$-invariant subset of $X_0$. 

Now, consider a projective completion $X$ of $X_0$, the extension $\varphi\colon X\dasharrow Y$ of $\varphi_0$, and a resolution of the indeterminacies of $\varphi$: this is given by a projective surface $Z$ and birational morphisms $\epsilon\colon Z\to X$, $\eta\colon Z \to Y$ such that $\varphi=\eta\circ \epsilon^{-1}$. 

There is an effective and ample divisor $H$ supported on $X\setminus X_0$ (see~\cite[Thm. 1]{Goodman:Annals}). 
Set $H_Z=\epsilon^*H$ and $H_Y=\eta_*(H_Z)$. If $D$ is an irreducible component of $H_Y$, and if $f_Y=\varphi\circ f\circ\varphi^{-1}$ is an element of $\Gamma_Y$, then $f_Y(D)$ is also a component of $H_Y$. Indeed, if $f_Y(D)$ is not contained in $\eta_*(H_Z)$, its strict transform in $Z$ under $\eta^{-1}$ is a curve $D_Z$ that intersects $\epsilon^{-1}(X_0)$; then $\epsilon(D_Z)$ gives a curve in $X_0$ or an indeterminacy point of $\varphi_0$. In both cases $\epsilon(D_Z)$ is mapped by $f^{-1}$ into $X\setminus X_0$, and this contradicts the fact that $f$ is in $\Aut(X_0)$. Thus, $\Gamma_Y$ preserves the support of $H_Y$, and a finite index subgroup of $\Gamma_Y$ preserves individually each component of $H_Y$. 

Since $H$ is ample, $H_Y$ is big and nef. Hence, $\Gamma_Y$ preserves a big and nef divisor. This implies that $\Gamma_Y$ is contained in an algebraic subgroup of $\Aut(Y)$. Consequently, $\Gamma$ is contained in an algebraic subgroup of $\Aut(X_0)$.(\footnote{Let us recall the proof of this last step (see~\cite[Sec. 2]{Cantat:Milnor} for references). The group $\Aut(Y)$ acts on the Néron-Severi group $\NS(Y)$, the connected component of the identity $\Aut(Y)^0$ is an algebraic group, and if $[A]$ is  a class of positive self-intersection, then $\Aut(Y)^0$ has finite index in the stabilizer of $[A]$ in $\Aut(Y)$. On the other hand, $\Gamma_Y$ preserves a big and nef class, and the self-intersection of such a class is automatically positive. Thus, $\Aut(Y)^0\cap \Gamma_Y$ has finite index in $\Gamma_Y$, which shows that $\Gamma_Y$ is contained in a finite extension of $\Aut(Y)^0$.})
\end{proof}

\subsection{Non-elementary groups}
 There is a natural hyperbolic space $\Hyp_X$ associated with $X$ on which $\Bir(X)$ acts by isometries. The notions of elliptic, parabolic, or loxodromic birational transformations coincide with the three possible kinds of isometries (see~\cite[\S{3.5}]{Cantat:Annals} and~\cite[Thm. 4.6]{cantat:cremona_survey}). Similarly, there is a notion of non-elementary group of isometries, which can be described without any reference to $\Hyp_X$:
a subgroup $\Gamma$ of $\Bir(X)$ is non-elementary if it contains a non-abelian free group, the element of which (except the identity) are all loxodromic. When $\Gamma$ contains a parabolic element $f$, $\Gamma$ is non-elementary if and only if it does not preserve the $f$-invariant fibration (see~\cite[Prop. 6.12]{Cantat:Annals} and~\cite[\S 6]{cantat:cremona_survey}). 

\subsection{Automorphisms of affine surfaces, characterizations of $\Gm\times\Gm$}\label{par:jonquieres_gmxgm}
Let $X_0$ be an affine surface. Let $X$ be a projective completion of $X_0$. We obtain an embedding $\Aut(X_0)\subset \Bir(X)$, 
and the above classification can be applied to the elements of $\Aut(X_0)$. The following results are proven in~\cite[Prop. 4.4.20]{Abboud:These} and~\cite[Cor. 10.11]{Abboud:IMRN}. 

\begin{thm}\label{thm:Abboud_invariant_curve}
Let $X_0$ be an affine surface and $g$ be a loxodromic automorphism of $X_0$. Then, $g$ does not preserve any curve $C\subset X_0$.
\end{thm}

\begin{thm}\label{thm:Iitaka-Abboud}
Let $X_0$ be an affine surface. Assume that {\rm{(i)}} $\Aut(X_0)$ contains a loxodromic automorphism and {\rm{(ii)}} there is a non-constant 
and non-vanishing function $\xi\colon X_0\to \bfK^\times$. Then, $X_0$ is isomorphic to $\Gm\times\Gm$ over $\bfKbar$. 
\end{thm}

Our next result concerns parabolic automorphisms $f\colon X_0\to X_0$. With the notation from Section~\ref{sec:elliptic-parabolic-loxodromic}, in cases (b) and (b'), we shall denote by $\pi_0\colon X_0\dasharrow B$ the rational map $\pi_0:=(\pi_Y\circ \varphi^{-1})_{\vert X_0}$; its fibers are permuted by $f$. If $X_0$ and $f$ are defined over a field $\bfK$, then $\pi_0$ is defined on a finite extension of $\bfK$.

\medskip
 
\begin{thm-C} 
 Let $X_0$ be a normal, affine surface defined over a field $\bfK$ of characteristic $0$. 
Let $f$ be a parabolic automorphism of $X_0$, let $\pi_0\colon X_0\dasharrow B$ be its invariant fibration, and let $f_B\in \Aut(B)$ be the automorphism induced by $f$ and $\pi_0$. Then,  
\begin{enumerate}[\rm (1)]
\item $f$ is a Jonquières twist,
\item $\pi_0$ does not have base points in $X_0(\bfKbar)$,
\item  if $f_B$ has infinite order, then there is a non-constant and non-vanishing function $\xi\colon X_0\to \bfK^\times$. 
\end{enumerate}
Moreover,  if $f_B$ has infinite order then $X_0$ is smooth.
\end{thm-C}

As we shall see below, Jonquières twists exist on some regular and some singular affine surfaces, for instance, on singular Markov surfaces (see~\S\ref{par:cayley_cubic}). 
 
 \medskip

\noindent\emph{Preliminary remarks.} To prove Theorem~C, we assume that $\bfK$ is algebraically closed. Indeed, if $\xi$ is a non-vanishing function defined over some Galois extension $\bfL$ of $\bfK$, then the product of its Galois conjugate provides a non-vanishing function defined over $\bfK$.

In the proof, we shall freely use some known results on rational fibrations of surfaces. In particular, as explained in~\cite{beauville:surfaces_asterisque}, if $Z$ is a smooth projective surface and $\pi\colon Z\to B$ is a fibration whose general fiber is a smooth rational curve, then there are no multiple fibers, and every singular fiber $F=\sum_i n_i C_i$ is 
made of smooth rational curves $C_i$ of negative self intersection, with at least one having self intersection $-1$. In particular, by blowing down such $-1$ curves and iterating this process, one obtains a relatively minimal surface $W$ with a ruling $\pi_W\colon W\to B$. 

\begin{proof}[Proof of Theorem~C]

Let $X$ be a projective completion of $X_0$. Denote by $\partial X$ the complement of $X_0$ in $X$; we may and do assume that $X$ is smooth near $\partial X$ and that $\partial X$ is a normal crossing divisor. By a theorem of Goodman~\cite{Goodman:Annals} (see p. 166, Cor. of its Thm. 1), $\partial X$ is connected. 

$\bullet$ Let us prove Assertion~(1) by contradiction. 

Otherwise, $f$ is a Halphen twist and according to Case (b') of Section~\ref{sec:elliptic-parabolic-loxodromic}, $f$ is birationally conjugate to an automorphism $f_Y$ of some projective surface~$Y$.
This contradicts Proposition~\ref{pro:not_conjugate_automorphism_projective}, and $f$ is therefore a Jonquières twist. 
 
 \smallskip

$\bullet$ The rational map $\pi_0$ extends to a rational map $\pi\colon X\dasharrow B$, where $B$ is a smooth projective curve. 
Let  $Z$ be a smooth projective surface and $\psi\colon Z\to X$ be a birational morphism on which $\pi_Z:= \pi\circ \psi$ is regular. The surface $Z$ is obtained from $X$ by a sequence of blow-ups, and we choose $Z$ to minimize this number of blow-ups. 
Doing so, $f$ lifts to a
birational transformation $f_Z$ of $Z$ that preserves the Zariski open subset 
$$\U_Z=\psi^{-1}(X_0)$$ 
and acts on it as an automorphism. The total transform of $\partial X$ in $Z$ will be denoted by $\partial Z$.  

\begin{rem}\label{rem:base_points_jonquieres_fibrations} If $\pi_0$ (or more precisely the pencil defined by $\pi_0$) had base points in $X_0$, 
then $\psi^{-1}$ would have indeterminacy points in $X_0$. 
We denote by $I_Z\subset \U_Z$ the total transform of the base points of $\pi_0$ by $\psi$. If non-empty, $I_Z$ is a curve that intersects the general fiber of $\pi_Z$ and is disjoint from $\partial Z$.\end{rem}

\begin{rem}
If $X_0$ has a singularity, then the singular locus of $X_0$ is a finite, $f$-invariant set, and its preimage in $Z$ is an invariant curve in $\U_Z$ (this curve can share some components with $I_Z$). 
\end{rem}

The general fiber of $\pi_Z$ is a smooth rational curve. Thus, as explained in the preliminary remarks,  there is a ruled surface $W$, with ruling $\pi_W\colon W\to B$, and a birational morphism $\eta\colon Z\to W$ such that $\pi_Z=\pi_W\circ \eta$. 
The fibers of $\pi_W$ are smooth rational curves of self-intersection $0$, and the fibers of $\pi_Z$ are obtained from them by a finite number of blow-ups.
Thus, if a fiber of $\pi_Z$ is not smooth, it is a tree of smooth rational curves of negative self-intersection. 

Let $E$ be an irreducible component of $\partial Z$. If $\pi_Z$ maps $E$ to a point, then $E$ is an irreducible component of a fiber and, as such, it is a smooth rational curve. Such a curve $E\subset \partial Z$ will be called a {\bf{vertical boundary component}}. 
If $C$ is an irreducible component of a fiber that intersects $\U_Z$, we say that $C$ is an {\bf{internal component}}. 

\smallskip

$\bullet$ Now, we modify $Z$ as follows. If there is a vertical boundary component of self-intersection $-1$, then we contract it. This gives a new surface $Z'$, with a birational morphism $Z\to Z'$ which is an isomorphism from $\U_Z$ to its image, denoted by $\U_{Z'}$; and $f_Z$ induces a birational transformation $f_{Z'}$ of $Z'$ that preserves $\U_{Z'}$, acting as an automorphism on it. The fibration $\pi_Z$ induces a fibration $\pi'\colon Z'\to B$.
By definition, the complement of $\U_{Z'}$ is the boundary $\partial Z'$. 

We repeat this process until we obtain a surface $Y$, with a fibration $\pi_Y\colon Y\to B$, and a birational morphism 
$\eta'\colon Z\to Y$ such that 
\begin{enumerate}[(a)]
\item $\eta'$ contracts only vertical boundary components and is an isomorphism from $\U_Z$ to its image, which we denote by $\U_Y$;
\item $\pi_Z=\pi_Y\circ \eta'$ and $\pi_Y$ does not have any vertical boundary component with self intersection $(-1)$;
\item $\eta'$ conjugates $f_Z$ to a birational transformation $f_Y$ of $Y$ that preserves the fibration $\pi_Y$ (permuting its fibers) and preserves the open set $\U_Y$, acting regularly on it. 
\end{enumerate}
Again, we denote by $\partial Y$ the complement of $\U_Y$. We can moreover assume (changing $W$ if necessary) that there is a birational morphism $\eta''\colon Y\to W$ such that $\eta=\eta''\circ \eta'$.
Then,
\begin{enumerate}[(a)]
\item[(d)] a vertical boundary component $E\subset \partial Y$ is either a smooth rational curve of self-intersection $0$, in which case it coincides with a fiber of $\pi_Y$, or a smooth rational curve with self-intersection $\leq -2$, in which case the fiber $F_b=\pi_Y^{-1}(b)$ containing $E$ is a tree of rational curves that contains at least one internal component with self-intersection $-1$;
\item[(e)] $f_Y$ permutes the internal components of the fibers (none of them is contracted by $f_Y$).
\end{enumerate}

$\bullet$ We say that a fiber $F_b=\pi_Y^{-1}(b)$ of $\pi_Y$ is {\bf{mobile}} if the orbit of $b$ under $f_B$ is infinite. 
Let us prove that 
\begin{enumerate}[(a)]
\item[(f)] mobile fibers are smooth and irreducible, and do not contain indeterminacy points of $f_Y$ or $f_Y^{-1}$;
they do not contain vertical boundary components.
\item[(g)] If a fiber $F_b$ contains an internal component, then it does not contain any indeterminacy points (for any iterate of $f_Y$).
\end{enumerate}
 
To prove (f), let $F_b$ be a mobile fiber. Since the $f_B$-orbit of $b$ is infinite, there is a positive integer $n$ such that $F_{f_B^n(b)}$ is a smooth, irreducible, internal curve that avoids indeterminacy points of $f_Y$. The proper transform of $F_{f_B^n(b)}$ 
by $f_Y^{-n}$ is an internal component $C$ of $F_b$, and is the only one. Thus, by (d), either $F_b$ coincides with it, and is smooth and irreducible, or $C^2=-1$ and the other irreducible components $E_i$ of $F_b$ are contained in $\partial Y$. In the latter case, $f_Y^n$ should contract the $E_i$. But, writing a resolution 
of the indeterminacies of $f_Y^n$ along $F_b$, one sees that the first curve $f_Y^n$ should contract is $C$, since by (d) the $E_i$ have self-intersection $\leq -2$. This is a contradiction, because $C$ is internal. Hence, $F_b$ is smooth, irreducible, and is equal to the internal curve $C$.

This result can be applied to $F_{f_B^n(b)}$ for every $n\in \Z$. If $F_b$ contained an indeterminacy point of $f_Y$, $F_{f_B(b)}$ would contain a boundary component or would not be smooth, and we would obtain a contradiction. Therefore, $F_b$ does not contain any indeterminacy point of $f_Y$ (resp. of $f^{-1}$).

To prove (g), we can now assume that $F_b$ is not mobile, hence that it is fixed. We argue as before: in a minimal resolution of the indeterminacies of $f_Y$ along $F_b$, the only curve one could be contracted would be the strict transform of an internal component, and this is not allowed.

\begin{rem}\label{rem:base_points_jonquieres_fibrationsII}  
As in Remark~\ref{rem:base_points_jonquieres_fibrations}, suppose that $\pi_0$ has a base point. Let $I_Y$ be the image of $I_Z$ in $Y$.
The curve $I_Y$ intersects the general fiber of $\pi_Y$, and hence all of them. Since $I_Y$ is contained in $\U_Y$, every fiber contains an internal component. By Property (g), $f_Y$ has no indeterminacy points, which contradicts $f_Y$ being a Jonquières twist (see Item~(b) in Section~\ref{sec:elliptic-parabolic-loxodromic}). 
Thus, {\sl{if $f$ is an automorphism of an affine surface $X_0$ and is a Jonquières twist, then its invariant fibration $\pi_0$ is regular on $X_0$ (it does not have base points)}}. This proves Assertion~(2) of the theorem.
\end{rem}

$\bullet$ Up to now, we have not assumed $f_B$ to be of infinite order. We now make this hypothesis and identify $B$ and $f_B$:

\begin{enumerate}[(a)]
\item[(h)] $B$ is a rational curve, $f_B$ is conjugate to a translation $\vb\mapsto\vb+1$ or to a scalar multiplication 
$\vb\mapsto \alpha\vb$, where $\alpha$ is not a root of $1$. In particular, $f_B$ has either $1$ or $2$ fixed points, and the
orbit of a point that is not fixed is infinite.
\end{enumerate}
 
Indeed, if $B$ is not a rational curve, then it is a curve of genus $1$ and every orbit of $f_B$ is infinite. Thus, 
$Y$  does not contain any vertical boundary component (by (f)), 
and  $f_Y$ is regular (again by (f)). On the other hand, a Jonquières twist is never a regular automorphism (see Item~(b) in Section~\ref{sec:elliptic-parabolic-loxodromic}). Property (h) follows from this contradiction.

As a consequence, with Property~(f) we get

\begin{enumerate}[(a)]
\item[(i)] a fiber of $\pi_Y$ which is not above a fixed point of $f_B$ is mobile; it is smooth and irreducible, and it does not contain any indeterminacy point (of $f^n$, for any $n\in \Z$). 
\end{enumerate}

$\bullet$ Since every fiber of $\pi_0$ intersects $\partial X$, there is a component $D$ of $\partial Z$ such that $\pi_Z$ maps $D$ onto $B$. We continue to denote by $D$ its image in $Y$. Since $f_Y$ preserves the fibration, it cannot contract
$D$: $f_Y$ permutes the components of $\partial Y$ which are mapped onto $B$ by $\pi_Y$. Changing $f_Y$ into 
some positive iterate $f_Y^k$, we may assume that these components are fixed by $f_Y$.
Doing so, $f_Y$ determines an automorphism $f_D$ of $D$, and the restriction of $\pi_Y$ to $D$ is a dominant morphism $\pi_D\colon D\to B$
 such that $\pi_D\circ f_D=f_B\circ \pi_D$. 

\begin{enumerate}[(a)]
\item[(j)] The ramification points of the morphism $\pi_D\colon D\to B$ are mapped to fixed points of $f_B$.
If $E$ is a reduced curve that is a component of a fiber (with some multiplicity $\geq 1$), then $E$ is transverse to $D$.
\end{enumerate}

Indeed, suppose that $D$ is tangent to the fibration at some point $q$. If $\pi_Y(q)$ is not a fixed point of 
$f_B$, then $f_Y^n$ is regular on a neighborhood of $q$ for all $n\in \Z$ (by Property~(i)), and the orbit of $q$ is 
infinite. Thus, $D$ should be tangent to the fibration on an infinite set, a contradiction. Thus, $\pi_Y(q)$ is a fixed point of $f_B$. The same argument applies to ramification points of $\pi_D$. 


If $f_B(\vb)=\vb+1$, then $f_D$ is also a translation, and $\pi_D$ has a unique ramification point. This is a contradiction since a map from $\P^1$ to $\P^1$ which is not étale must have at least $2$ ramification points.
Thus, $\pi_D$ is an isomorphism onto $B$ (because $B\simeq \P^1$), and $D$ is everywhere transverse to the fibration.

If $f_B(\vb)=\alpha\vb$, then $f_D$ is also a scalar multiplication, and $\pi_D$ has $2$ ramification points, one above $0$ and one above $\infty$. Note that one of the fibers $F_0$ or $F_\infty$ is contracted by $f_Y$, because $f_Y$ is a Jonquières twist (see, again, Item~(b) in Section~\ref{sec:elliptic-parabolic-loxodromic}).
 On the other hand, 
if a fiber $F$ is contracted by a Jonquières twist, then any invariant irreducible curve that is generically transverse to the 
fibration must intersect $F$ transversally (see Section~3, and in particular Lemma~3.28 in \cite{Zhao:Pisa}). This concludes the proof of Property~(j). 


\smallskip

$\bullet$ If $f_B(\vb)=\vb+1$ the fiber $F_\infty$ is the only fiber contracted by $f_Y$. From Properties (d) and  (g), this fiber is a smooth boundary component. Thus, by Property (f), $Y$ is ruled. 
Doing elementary transformations along $F_\infty$, we do not modify $\U_Y$ and we can construct a birational 
map $\epsilon \colon Y\to \P^1\times \P^1$ that transforms $\pi_Y$ to the first projection, maps $D$ to the section $\P^1\times \set{\infty}$, and transforms $F_\infty$ to $\set{\infty}\times \P^1$. Then, using affine coordinates $(\vx,\vy)$ on $\P^1\times \P^1$, $\epsilon$ conjugates $f_Y$ (i.e.\ $f_Y^k$) to a map of type 
$$(\vx,\vy)\mapsto (\vx+1, a(\vx)\vy+b(\vx))$$
where $a(\vx)$ is some regular function of $\vx\in \A^1$ that does not vanish, since otherwise a mobile fiber would be contracted. Thus, $a$ is a constant, and this contradicts the fact that $f_Y$ is a Jonquières twists (the degree of the iterates would be bounded). So, the case $f_B(\vb)=\vb+1$ does not appear.

\smallskip

$\bullet$  If $f_B(\vb)=\alpha\vb$, we apply the same strategy. At least one fiber among $F_0$ and $F_\infty$ is contracted, say $F_\infty$; it is a smooth irreducible component of $\partial Y$. 

Assume first that $F_0$ does not contain any internal component. Then it is smooth and contained in $\partial Y$. 
Doing elementary transformations centered on $F_\infty$, we get a map $\epsilon \colon Y\to \P^1\times \P^1$ that
maps $\U_Y$ into $(\P^1\setminus\set{0,\infty}) \times (\P^1\setminus \set{\infty}))$ and conjugates $f_Y^k$ to 
a map of type 
$$ (\vx,\vy)\mapsto (\alpha\vx, a(\vx)\vy+b(\vx))$$
where $a$ is non-constant and does not vanish on $\P^1\setminus\set{0,\infty}$. Thus, $a\circ \epsilon$ is the desired non-constant, non-vanishing function we were looking for. 

Assume now that $F_0$ contains an internal component. Then $f_Y$ acts regularly on $Y\setminus F_\infty$.
Let us contract the $(-1)$-components of $F_0$ iteratively to reach a ruled surface $W$; this process may modify 
$\U_Y$. With elementary transformations based on $F_\infty$, we construct a birational map $\epsilon\colon W\to \P^1\times \P^1$ as above. It maps $D$ to $\P^1\times\set{\infty}$ and conjugates $f_Y$ to a map of type 
$$ (\vx,\vy)\mapsto (\alpha\vx, a(\vx)\vy+b(\vx))$$ 
where $a$ is not constant and does not vanish if $\vx\neq 0$. Thus, 
$a(\vx)=\beta\vx^d$ for some $d\neq 0$ and $\beta\neq 0$. Since $d\neq 0$, the fiber $\set{0}\times \P^1$ is 
contracted, hence so is $F_0$. Thus, in fact, $F_0$ is smooth and contained in $\partial Y$, contradicting our hypothesis.

\begin{rem} We have excluded the case $f_B(\vb)=\vb+1$ and we have proven that $F_0$ and $F_\infty$ do not contain any internal components. Therefore, every point of $\U_Y$ has an infinite orbit (since its projection under $\pi_Y$ also has an infinite orbit). In particular, $X_0$ is smooth, since otherwise a singularity would create a fiber of $\pi_Y$ with a finite orbit intersecting $\U_Y$. This remark proves the last sentence of Theorem~C.\end{rem}

\smallskip

$\bullet$ To conclude, we need to go back to $X_0$. In the last step, we have constructed a regular function $\xi:=a\circ \epsilon$ on $\U_Y$ that is non-constant and does not vanish. But $\U_Y$ is isomorphic to $\eta'(\psi^{-1}(X_0))$, and 
$\psi\circ(\eta')^{-1}\colon Y\dashrightarrow Z\to X$ is obtained by a sequence of blow-ups (of the base points of $\pi_0$). By construction, $\xi\circ\eta'\circ \psi^{-1}$ is regular 
and does not vanish outside the indeterminacy points of $\psi\circ(\eta')^{-1}$. Thus, it extends to a regular, non-vanishing function on $X_0$ (because $X_0$ is normal). This concludes the proof of Theorem~C. 
\end{proof}

 \begin{cor} Let $X_0$ be an affine surface. Assume $\Gamma\subset \Aut(X_0)$ is non-elementary and contains a parabolic element acting 
 as an automorphism of infinite order on the base of its invariant fibration. Then, $X_0$ is isomorphic to $\Gm\times\Gm$ over $\bfKbar$. 
 \end{cor}

Indeed, Theorem~C shows that there is a non-constant and non-vanishing regular function on $X_0$. Since $\Gamma$ contains a loxodromic element, Theorem~\ref{thm:Iitaka-Abboud} shows that $X_0$ is isomorphic to $\Gm\times \Gm$.

 \begin{cor}\label{cor:application2_of_thm_C} 
 Let $X_0$ be an affine surface that is not isomorphic to $\Gm\times\Gm$ over $\bfKbar$. Assume that 
 $\Aut(X_0)$ contains a loxodromic automorphism. If $f$ is a parabolic element of $\Aut(X_0)$ then 
 \begin{enumerate}[\rm(1)]
 \item $f$ preserves a regular fibration $\pi_0\colon X_0\to B$ with base and fibers of genus~$0$;
 \item  a positive iterate $f^k$ of $f$ preserves each fiber of this fibration;
 \item  the Zariski closure of every 
 orbit of $f$ has dimension $0$ (if the orbit is finite) or $1$ (if the orbit is infinite, in which case its closure is a finite number of fibers of $\pi_0$). 
 \end{enumerate}
 \end{cor}
 
Indeed, $\pi_0$ does not have base points in $X_0$ by Assertion~(2) in Theorem~C (or Remark~\ref{rem:base_points_jonquieres_fibrationsII}).

\subsection{Finite orbits of parabolic automorphisms} For a subring $R$ of $\bfK$, denote by $\Per_f(R)$ the set of periodic points of $f$ contained in $X_0(R)$.

\begin{thm}\label{thm:parabolic_periodic_algebraic}
Let $X_0$ be an affine surface defined over the valuation ring $\VR$ of a $p$-adic field $\bfK$. 
Let $f$ be a parabolic automorphism of $X_0$ defined over $\VR$. Then, there is an integer $m\geq 1$ 
and a Zariski closed subset $P_f$ of $X_0$ such that 
$ \Per_f(\VR)=P_f(\VR)=\set{x\in X_0(\VR)\; ; \; f^m(x)=x}.$
\end{thm}
Of course, $P_f$ and $m$ depend on $f$ and on $\bfK$. The same result does not hold on the algebraic closure of $\bfK$: see Example~\ref{eg:finite_orbits_monomial}.


\begin{lem}
Let $f$ be an automorphism of an irreducible, affine curve $F\subset \A^N$. If $x\in F$ is $f$-periodic, then 
either $f$ has finite order in $\Aut(F)$ or $f(x)=x$. 
\end{lem}

\begin{proof} Let $\overline{F}$ be the completion of $F$ in $\P^N$, and let $s$ be the number of points of $\overline{F}$ at infinity. 
Let $r$ be the size of the orbit of $x$. If $x$ is not fixed, we have $r\geq 2$ and $s\geq 1$ so $r+s\geq 3$. If we set $m=(r+s)!$, then $f^m$ fixes at least $3$ points of $\overline{F}$. Thus, $f^m=\Id_F$.  
\end{proof}

\begin{proof}[Proof of Theorem~\ref{thm:parabolic_periodic_algebraic}]
Let $\pi_0\colon X_0\to B$ be the $f$-invariant fibration, and let $f_B$ be the automorphism of $B$ induced by $f$. 

If $f_B$ has infinite order, then $f_B$ has at most two periodic points, which are fixed. Denote them by $\set{b, b'}$. Then, the finite orbits of $f$ are contained in the fibers $\pi_0^{-1}(b)$ and $\pi_0^{-1}(b')$, and  the theorem follows from the previous lemma, applied to the restriction of $f$ to $\pi_0^{-1}(b)$ and $\pi_0^{-1}(b')$.

Now, assume that $f_B$ is an element of order $1\leq \ell < \infty$ in $\Aut(B)$. Then, changing $f$ to $f^\ell$, we are reduced to the case $f_B=\Id_B$. Let $k$ be a positive integer such that $f^k$ fixes every irreducible component of every fiber of $\pi_0$. From 
the previous lemma, the problem is to show that the number of points $b\in B(\VR)$ such that $f_{\vert F_b}$ has finite order is finite. For this, we can reduce the study to regular fibers. But, in some completion $X$ of $X_0$, every regular fiber has a neighborhood $U$ isomorphic to $U_0 \times \P^1$ (for some $U_0\subset B$, see~\cite[Thm. III.4]{beauville:surfaces_asterisque}), on which $f$ acts by 
\begin{equation}
f(b,z)=(b, A_b(z))
\end{equation}
where $A\colon b\in B\mapsto A_b\in \PGL_2(\bfKbar)$ is a non-constant rational map. Both $\pi_0$ and $A$ are defined on 
a finite extension $\bfL$ of $\bfK$. So, to conclude, we only need to show that the set 
\begin{equation}
\set{b\in B(\VRL)\; ; \; A_b \; \text{has finite order in} \; \PGL_2(\bfL)}
\end{equation}
is finite. To see this, we lift $A_b$ to an element $A'_b$ in $\GL_2(\overline{\bfL})$ and set 
\begin{equation}
r(b)=\frac{\Tr(A'_b)^2}{\det(A'_b)}-2;
\end{equation}
then the order of $A_b$ is finite if and only if $r(b)=\alpha+\alpha^{-1}$ for some root of unity $\alpha\in \overline{\bfL}$; more precisely, $A_b$ is conjugate to $z\mapsto \alpha z$ and the order of $\alpha$ is the order of $A_b$. 
When $b$ is in $B(\VRL)$, $r(b)$ is in $\bfL$ and $\alpha$ is a root of unity in $\overline{\bfL}$ that satisfies a 
quadratic equation over $\bfL$. The structure of the multiplicative group of finite extensions of $\Q_p$ (see~\cite{Neukirch}, p.140, or Section~\ref{par:multiplicative_groups_arithmetic} below) implies that the order of such a root of unity is uniformly bounded, independently of $r(b)\in \bfL$. This concludes the proof. 
\end{proof}

\begin{rem}\label{rem:L_for_parabolic}
If $f_B=\Id_B$ and, in some invariant bidisk $\U\simeq \VR^2$, $f=\Phi^1$ for some flow $\Phi^t\colon \VR\times\U\to \U$, then 
the vector field $\Theta_f=\partial_t\Phi^t_{\vert t=0}$ is everywhere tangent to the fibers of $\pi_0$. Thus, we may extend the distribution of lines $x\mapsto \bfK \Theta_f(x)$ globally as 
the algebraic distribution of tangent lines $x\mapsto \Ker(d\pi_{0,x})$. Note that $\Theta_f$ vanishes identically on each fiber $F_b$ on which $f$ induces a finite order automorphism. 
\end{rem}


\section{Elliptic automorphisms}


\subsection{Bounded degrees}\label{par:bounded_degree_algebraic_completion} Let $f$ be an automorphism of an affine variety $V$, both defined over a field $\bfK$ of characteristic $0$.

If $(\deg(f^n))$ is bounded, then $f$ is contained in 
a linear algebraic group $G$ acting algebraically on $V$ (see~\cite[Prop. 3.2]{cantat-regeta-xie}, with $B$ as a point). 
The Zariski closure of $f^\Z$ in $G$ is an abelian algebraic subgroup $A$ of $G$; we call it the \emph{algebraic completion of $f^\Z$}. Let $A^\circ$ denote the (Zariski) connected component of the identity in $A$, and let $k$ be the index of $A^\circ$ in $A$; then, $f^k$ is 
contained in $A^\circ$, and its iterates form a Zariski dense subset of $A^\circ$. In what follows, we assume that $\dim(A^\circ)\geq 1$; equivalently, the order of $f$ in $\Aut(X_0)$ is infinite. 

On the algebraic closure $\bfKbar$ of $\bfK$, $A^\circ$ is isomorphic to $\Ga^r\times \Gm^s$ for some integers $r,s\geq 0$ (here we use ${\mathrm{char}}(\bfK)=0$).
Since the cyclic group $(f^k)^\Z$ is Zariski dense in $A^\circ$, we have $r\leq 1$.  
The algebraic subgroups of $\Gm^s$ form a countable family of subgroups, so the generic point of $X_0$ has a trivial stabilizer in $\Gm^s\subset A^\circ$. Since the only nontrivial algebraic subgroup of $\Ga$ is $\Ga$ itself, the general point of $X_0$ also has a trivial stabilizer in $\Ga\subset A^\circ$. Thus $r+s\leq \dim(V)$.

\subsection{The Lie algebra of $A^\circ$ and the Bell-Poonen theorem}\label{par:bell_poonen_elliptic}
Taking the derivative of the action of $A^\circ$ on $V$, the Lie algebra $\Lie(A^\circ)$ determines a subalgebra of the Lie algebra $\Gamma(V,TV)$ of regular vector fields on $V$. We identify $\Lie(A^\circ)$ with its image in $\Gamma(V,TV)$. 

Assume, now, that $\bfK$ is a $p$-adic local field, and that $V$ and $f$ are defined over its valuation ring $\VR$. As in the Bell-Poonen theorem, suppose 
 there exists an open set $\U\simeq \VR^2\subset V(\VR)$ preserved by $f$ on which $f=\Phi^1$ for some analytic flow $\Phi\colon \VR\times \U\to \U$. Then, $\Phi^t$ corresponds to a $1$-parameter subgroup of $A^\circ$. In particular, there is an algebraic vector field $a_f\in \Lie(A^\circ)$ such that 
 \begin{equation}
  \Theta_f(y)=\left( \frac{\partial \Phi^t(y)}{\partial t} \right)_{t=0}=a_f(y)
\end{equation}
for $y\in \U$. 
Thus, $\Theta_f$ extends from $\U$ to $V$ as a globally defined algebraic vector field. 

For each $y$ in $\U$ (resp. $V$), we denote by $L_f(y)$ the line $\bfK a_f(y)\subset T_yV$, with a small abuse of notation because this distribution of lines is not well-defined at the points $y$ where $a_f(y)=0$. (Here, in comparison with \S\ref{par:lie_dimension}, the lines are defined over $\bfK$ instead of $\Q_p$.) We have shown the following.

\begin{lem} 
	\label{lem:elliptic-generator-algebracity}
If $f$ is an elliptic automorphism, the distribution of tangent lines 
$$L_f\colon y\mapsto L_f(y):=\bfK\Theta_f(y)=\bfK a_f(y)$$ 
depends algebraically on $x$. In other words, it defines a rational section of the projectivized tangent bundle $\P(TV)$.
\end{lem}

Let $x$ be a point of $V(\VR)$ and set $\gamma(t)=\Phi^t(x)$. Let  $g$ be an element of $\Aut(V_{\VR})$. We say that $g$ {\bf{preserves locally the analytic trajectory}} of $x$ determined by $f$ if $g$ maps $\gamma(\mK^\ell)$ into $\gamma(\VR)$ for some $\ell \geq 1$.
There we obtain 
\begin{equation}
g(\gamma(t))=\gamma(\varphi(t))
\end{equation}
for some germ of analytic function $\varphi$ around the origin in $\VR$. Taking derivatives, we get
\begin{equation}
Dg_{\gamma(t)}{\dot\gamma}(t)=\varphi'(t){\dot\gamma}(\varphi(t))
\end{equation}
where $Dg$ is the differential, ${\dot\gamma}$ is the velocity vector of the curve $\gamma$, and $\varphi'$ is the derivative of the function $\varphi$. Thus, along the curve $\gamma$, $g$ preserves the distribution of tangent lines to the orbits of $\Phi^t$. But this distribution of lines $x\mapsto \bfK\Theta_f(x)$ depends algebraically on $x$. This gives the following lemma.

\begin{lem}\label{lem:invariance_of_Lf}
 Suppose $g\in \Aut(V_{\VR})$ preserves the analytic trajectory of $x$, and that this trajectory is Zariski dense in $V$. Then, the distribution of lines $$L_f\colon y\in V\mapsto L_f(y):=\bfK a_f(y)$$ is $g$-invariant. 
\end{lem}

\subsection{Finite orbits}  
\begin{lem}\label{lem:finite_orbits_of_elliptic}
Let $f$ be an elliptic automorphism of an affine variety $V$, let $A^\circ$ be the connected algebraic group associated with $f$, and let $P_f$ be the algebraic set defined by $P_f=\set{x\in V\; ; \; a(x)=0\; \; \forall a\in \Lie(A^\circ)}$. Then
 $q\in V(\bfKbar)$ has a finite $f$-orbit if and only if $q\in P_f(\bfKbar)$;
there is an integer $\ell$ such that the period of every point $q\in P_f(\bfKbar)$ divides $\ell$. 
\end{lem}

Note that this lemma works for all finite orbits in $V(\bfKbar)$, in contrast with Theorem~\ref{thm:parabolic_periodic_algebraic}.

\begin{proof} We may assume that the order of $f$ is infinite, since otherwise $P_f=V$ and we can set $\ell$ to be the order of $f$. 
A point $q\in V(\bfKbar)$ has a finite orbit if and only if its stabilizer has finite index in $f^\Z$. Since such a subgroup intersects $A^\circ$ on a Zariski dense subgroup, $q$ has a finite orbit
if and only if $A^\circ(q)=q$, if and only if $q\in P_f(\bfKbar)$. We can then take $\ell$ to be the index of $A^\circ$ in $A$.
\end{proof}

\begin{eg} \label{eg:large-fixed-locus}
	If $f$ is the automorphism of the plane defined by $f(\vx,\vy)=(\vx+\vy^2+\vy^3, -\vy)$, then $f$ is elliptic. Indeed, $f^{2n}(\vx,\vy)=(\vx+2n\vy^2,\vy)$ has degree $2$, and $f^{2n+1}=(\vx+(2n+1)\vy^2+\vy^3,-\vy)$ has degree $3$. This reveals that $P_f=(\vy^2=0)=(\vy=0)$. We can also explicitly compute the algebraic completion $A$ of $f$: the identity component $A^\circ$ acts by $\Phi^t(\vx,\vy)=(\vx+t\vy^2,\vy)$; its second component is obtained by the $f$-shift $f\circ\Phi^t(\vx,\vy)=(\vx+(t+1)\vy^2+\vy^3,-\vy)$. One can further generalize this to elementary transformations of the type $(\vx,\vy)\mapsto(\alpha\vx+q(\vy),\beta\vy)$, where $\alpha$ and $\beta$ are roots of unity.
\end{eg}

\subsection{Structure of $A^\circ$}\label{par:structure_algebraic_completion}  Let us apply Sections~\ref{par:bounded_degree_algebraic_completion} and~\ref{par:bell_poonen_elliptic} when $f$ is an elliptic automorphism of an affine surface $X$.  We obtain $A^\circ=\Ga^r\times\Gm^s$ with $r\leq 1$ and $r+s\leq 2$. More precisely,  over $\bfKbar$, exactly one of the following situations occurs:
\begin{enumerate}[(1)]
\item\label{enum:alg-compl-GmxGm} $r=0$ and $s=2$, and $A^\circ=\Gm\times \Gm$ has a unique open orbit, which is isomorphic to$~A^\circ$;
\item $r=1$ and $s=1$, and $A^\circ=\Ga\times \Gm$ has a unique open orbit, which is isomorphic to$~A^\circ$;
\item $r=0$ and $s=1$, or $r=1$ and $s=0$, then $A^\circ = \Gm$ or $\Ga$, respectively. The general orbits of $A^\circ$ have dimension $1$, and they define a rational fibration $\pi_0\colon X\dasharrow B$ onto a curve $B$ such that $\pi_0\circ h=\pi_0$ for every $h\in A^\circ$.
\end{enumerate}
If $f$ has a Zariski dense orbit, we are in one of the first two cases. 

\begin{eg} 
If $f$ is the automorphism of the plane defined by $f(\vx,\vy)=(\vx+1, \alpha \vy)$, where $\alpha\in \bfK^\times$ is not a root of unity, then the orbit of $(x,y)$ is Zariski dense in $\A^2$ if and only if $y\neq 0$. If $f(\vx,\vy)=(\alpha \vx,\beta \vy)$ and $\alpha$ and $\beta$ are elements of $\bfK^\times$ that are multiplicatively independent, then the orbit of $(x,y)$ is Zariski dense if, and only if, $xy\neq 0$. 
\end{eg}

\subsection{The Fixed Locus}\label{par:fixed_locus_elliptic}

If a reductive algebraic group $G$ acts on an affine variety $V$, then any two disjoint, invariant, Zariski closed subsets can be separated by a regular invariant function on $V$ (see \cite[Cor. 1.2]{Mumford:GIT-1ere-ed}). 

\begin{lem}
    \label{lem:dim-1_dim-0_relation}
    Let $\Gm\times V\to V$ be a nontrivial algebraic action on an affine variety $V$. For each point $x\in V$, 
    \begin{enumerate}[\rm (1)]
        \item  the Zariski closure of the $\Gm$-orbit $\Gm(x)$ contains at most one fixed point;
        \item if $x$ is fixed by $\Gm$, there is a $\Gm$-orbit which is 1-dimensional and whose Zariski closure contains $x$.
    \end{enumerate}
\end{lem}
\begin{proof}
    (1) Fix $x\in V$. If the Zariski closure of $\Gm(x)$ has more than one $\Gm$-fixed point, then these points must be separated by a regular invariant function~\cite[Cor. 1.2]{Mumford:GIT-1ere-ed}. Such an invariant function must be constant on the Zariski closure, a contradiction. 

    (2) We proceed by induction on $\dim V$. If $\dim V=1$, then as $\Gm$ acts nontrivially on $V$, there is a dense orbit in $V$, and the closure of this orbit contains~$x$. 
    
    Now suppose $\dim V \geq 2$.  If every 1-dimensional $\Gm$-orbit closure contains the fixed point $x$, then we are done. Otherwise, suppose there exists $y\in V$, not fixed by $\Gm$, such that $x$ is not contained in the Zariski closure of $\Gm(y)$.

    Consider the GIT quotient $V/\!\!/\Gm$ (see, e.g., \cite[Def. 5.8]{Mukai:GIT}) and the natural projection $\pi\colon V\to V/\!\!/\Gm$. Observe that (a) the fiber $\pi^{-1}(\pi(x))$ contains a unique fixed point, namely $x$, and (b) the fiber $\pi^{-1}(\pi(x))$ has dimension $<\dim V$, as $\pi$ separates the fiber from $\Gm(y)$. On the other hand, the general orbit of $\Gm$ has dimension $1$, so all fibers of $\pi$ have positive dimensions. By choosing $V'\subset\pi^{-1}(\pi(x))$ to be an irreduible component containing $x$, we have $0<\dim V'<\dim V$. Since $\Gm$ is connected, $V'$ is invariant, and since $\pi$ separates fixed points, the action of $\Gm$ on $V'$ is nontrivial. Thus, we can conclude by induction.
\end{proof}

Using this, we can study the dimension of the fixed locus $P_f$, for $f$ an elliptic automorphism of an affine surface $X$. Let $A$ be the algebraic completion of $f^\Z$. By section \ref{par:structure_algebraic_completion}, we know that $A^\circ$ is either $\Gm$, $\Ga$, $\Ga\times\Gm$, or $\Gm\times\Gm$.
We have observed, in Example \ref{eg:large-fixed-locus}, that if $\dim A^\circ=1$ then $P_f$ may be a curve. 

\begin{pro}\label{pro:Pf_is_finite_if_dimA_is_2}
    If $\dim A^\circ=2$, then $P_f$ has at most one point.
\end{pro}
\begin{proof}
    $\bullet$ First, suppose $A^\circ\simeq\Ga\times\Gm$. Denote simply by $\Gm$ the subgroup $\{ 0\}\times \Gm \subset A^\circ$, and consider the quotient $\pi\colon X\to B:=X/\!\!/\Gm$. 
 By \cite[Thm. 1.1]{Mumford:GIT-1ere-ed}, $B$ is affine. As $\Gm$ acts nontrivially on $X$, the general fibers of  $\pi$  have dimension $1$ and $\dim B\leq 1$.

    If $\dim B=0$, then $B$ is a point because $X$ is irreducible; since $P_f$ is pointwise fixed by $\Gm$, and $\pi$ separates fixed points, there is at most one point in $P_f$.
    
    Suppose $\dim B=1$. The action of $A^\circ$ on $X$ induces an algebraic action of $\Ga$ on $B$. If $A^\circ$ has a fixed point $q$, then $\pi(q)$ is fixed by $\Ga$, and thus $\Ga$ acts trivially on $B$ (\footnote{If $\Ga$ acted non trivially on an irreducible affine curve $B$, it would have no fixed point. Indeed, $\Ga$ acts on the normalization $\tilde{B}$ of $B$, and since the action is nontrivial, $\tilde{B}$ must be a smooth, affine, rational curve. In other words, $\tilde{B}$ is isomorphic to $\P^1\setminus F$, for some finite subset $F$, with $\vert F\vert \geq 1$. But then, since $\Ga$ acts nontrivially on it, $\tilde{B}\simeq \A^1$. The action of $\Ga$ has no fixed point, and $\tilde{B}\simeq B$.}).
 %
So, we can assume that $\Ga$ acts trivially on $B$. 
Then $\Ga\times\Gm$ acts faithfully on the generic fiber of $\pi$ (i.e., $\pi^{-1}(\eta_B)$ where $\eta_B$ is the generic point of $B$; as the fiber is dense in $X$, we have the faithfulness). This leads to a contradiction, as the group $\Ga\times\Gm$ cannot act faithfully on a curve.

$\bullet$ Next, suppose $A^\circ\simeq\Gm\times\Gm$. 
\iftrue
As remarked in item (\ref{enum:alg-compl-GmxGm}) of \S{\ref{par:structure_algebraic_completion}}, we have a unique open orbit $A^\circ(x)$ isomorphic to $A^\circ$ in $X$, which must be Zariski dense as well. %
\else
Observe the following facts:
\begin{enumerate}[\rm (1)]
\item ``upper semicontinuity of stabilizers'': for a fixed algebraic subgroup $G\subset A^\circ$, the set $X_G=\{x\in X : G(x)=\set{x}\}$ is a Zariski closed subset.
\item ``countably many subgroups'': there are countably many algebraic subgroups of $A^\circ$.
\end{enumerate}
For a nontrivial subgroup $G\subset A^\circ$, we have $X_G\neq X$ since otherwise this contradicts that $A^\circ$ acts faithfully on $X$. Together with the Baire category theorem, we conclude that the stabilizer of a very general point of $X$ is trivial. The orbit of such a point $x$ yields a dense, Zariski open subset of $X$.
\fi
Thus, any $A^\circ$-invariant regular function on $X$ must be constant, equal to the value taken at $x$. Therefore the GIT quotient $X/\!\!/A^\circ$ must be a point. Since $P_f$ coincides with the set of fixed points of $A^\circ$, and the quotient $X\to X/\!\!/A^\circ$ separates fixed points (by \cite[Cor. 1.2]{Mumford:GIT-1ere-ed} again), we conclude that $\vert P_f\vert \leq 1$.
\end{proof}

\begin{eg}
    The above proof shows that, if $A^\circ=\Ga\times\Gm$ then $A^\circ$ does not have a fixed point unless $\dim(X/\!\!/\Gm)=0$. 
To get an example, set $h^s=(\begin{smallmatrix} 1 & s \\ 0 & 1 \end{smallmatrix})$ for $s\in\Ga$, and let $\Ga\times\Gm$ act on $X=\A^2$ by $(\vs,\vt) \cdot(\vx,\vy)=\vt h^\vs(\vx,\vy)=(\vt(\vx+\vs\vy),\vt\vy)$. The point $(0,0)$ is fixed and is the only fixed point of $A^\circ$ in $X$.
\end{eg}

\subsection{Elliptic automorphisms with Zariski dense orbits} 




\begin{thm}\label{thm:GmGm_and_invariant_analytic_curve}
Let $X$ be a normal affine surface defined over the valuation ring $\VR$ of a $p$-adic local field $\bfK$.
Let $f$ be an elliptic automorphism of $X$, defined over $\VR$. Let $\U\subset X(\VR)$ be an $f$-invariant bidisk, on which 
$f=\Phi^1$ for some analytic flow. If there is a loxodromic automorphism $g$ in $\Aut(X)$ defined over $\VR$ and a point 
$x\in \U$ such that (i) the analytic trajectory $t\mapsto \Phi^t(x)$ is Zariski dense in $X$ and (ii) $g$ preserves locally this analytic trajectory, then $X$ is isomorphic to $\Gm\times\Gm$ over $\bfKbar$.
\end{thm}

The proof is related to the classification of foliations on complex surfaces invariant by an infinite group of birational transformations (see~\cite[Thm. 1.1]{Cantat-Favre}).

\begin{proof} Let $\gamma$ be the curve $t\mapsto \Phi^t(x)$. By assumption,  $g$ is loxodromic and the image of $\gamma$ is Zariski dense. 
Hence, the algebraic completion $A$ of $f^\Z$ has $\dim(A)=2$ and $A^\circ$ is isomorphic to $\Ga\times \Gm$ or $\Gm^2$ over $\bfKbar$. Let $k$ be the index of $A^\circ\subset A$.

\smallskip

{\bf{Step 1.--}} By Lemma~\ref{lem:invariance_of_Lf}, $L_f$ is $g$-invariant, hence 
there is a rational function $\xi$ on $X$ such that $g_*{a_f}=\xi \cdot {a_f}$. This function $\xi$ is regular on the complement of the zero set 
\begin{equation}
Z(a_f)=\set{x\in X\; ; \; {a_f}(x)=0};
\end{equation} 
but Proposition~\ref{pro:Pf_is_finite_if_dimA_is_2} shows that this set is finite because it is pointwise fixed by $f^k$, hence by $A^\circ$. Since $X$ is normal, $\xi$ extends as a regular function on $X$. Moreover, if $\set{\xi=0}$ were non-empty, it would be a curve, and the differential $dg_x$ would have rank $\leq 1$ along a curve. However, since $g$ is an automorphism, we see that $\xi$ does not vanish at all. By Theorem~\ref{thm:Iitaka-Abboud}, we may assume that $\xi$ is a constant. 
This implies that $g$ conjugates $f^k$ to another 
element of $A^\circ$. Since $(f^k)^\Z$ is Zariski dense in $A^\circ$, the action of $g$ by conjugacy on $\Aut(X)$ preserves $A^\circ$. In other words, there is an algebraic automorphism $\psi_g$ of the algebraic group $A^\circ$ such that $ghg^{-1}=\psi_g(h)$ for every $h\in A^\circ$.

\smallskip

{\bf{Step 2.--}} If $A^\circ$ were isomorphic to $\Ga\times \Gm$, the action of $g$ on it by conjugacy would preserve the factor $\Ga$, and then $g$ would preserve the rational fibration of $X$ given by the orbits of \iftrue$\Gm$\else$A^\circ$\fi. This contradicts to the fact that $g$ is loxodromic (see Theorem \ref{thm:Abboud_invariant_curve}). Thus, $A^\circ$ is isomorphic to $\Gm^2$ over $\bfKbar$, and $\psi_g$ is a monomial automorphism of $\Gm\times\Gm$. From this, $X_{\bfKbar}$ contains a copy of $\Gm\times \Gm$, represented by the unique open orbit of $A^\circ$, which is $g$-invariant because $gA^\circ g^{-1}=A^\circ$, and $g$ acts on it 
via a loxodromic monomial transformation. This implies that $X$ is isomoprhic to $\Gm\times \Gm$. Indeed, 
let $Y$ be a projective completion of $X$. It is a projective completion of $\Gm\times \Gm$, and
$g$ induces a birational transformation of $Y$ acting as a regular automorphism on $\Gm\times \Gm$. 
Let $(\vx_1,\vx_2)$ be the affine coordinates on $\Gm\times \Gm$. 
There are elements $\alpha$, $\beta$ in $\bfKbar^\times$, and a matrix 
\begin{equation}
M_g=\left(\begin{array}{cc} a & b \\ c & d \end{array}\right)
\end{equation}
of $ \GL_2(\Z)$ such that $g(\vx_1,\vx_2)=(\alpha \vx_1^a \vx_2^b, \beta\vx_1^c \vx_2^d)$. Since $g$
is loxodromic, its dynamical degree is $>1$, and $M_g$ is conjugate in $\GL_2(\R)$ to a diagonal matrix 
with eigenvalues $\pm \lambda_1(g)$ and $\pm \lambda_1(g)^{-1}$, with eigenvectors in $\R^2\setminus \Q^2$. This implies that $Y\setminus (\Gm\times \Gm)$ is entirely contracted by some positive (resp. negative) iterate of $g$ on an indeterminacy point of $g^{-1}$ (resp. of $g$). Since $g$ induces an automorphism of $X$, this implies 
that $X=\Gm\times \Gm$.
\end{proof}

\section{The multiplicative group $\Gm^2$}\label{par:case_of_GmxGm}


\subsection{Multiplicative groups}\label{par:multiplicative_groups_arithmetic} Let $d$ be a positive integer, and let $V$ be the mutiplicative group $\Gm^d$, viewed as an affine variety. 
The automorphism group $\Aut(V)$ is $\GL_d(\Z)\ltimes V$, where $V$ acts on itself by tranlation
and $\GL_d(\Z)$ by monomial transformations: the action of a matrix $(a_{i,j})\in \GL_d(\Z)$ is 
given by
\begin{equation}
(v_1,\ldots, v_d)\mapsto (v_1^{a_{1,1}}\cdots v_d^{a_{d,d}},\, \ldots,\, v_1^{a_{d,1}}\cdots v_d^{a_{d,d}}).
\end{equation}

Let $p$ be a prime and $\bfK$ be a $p$-adic local field, with uniformizer $\pi$, valuation ring $\VR$, and group of principal units $U_\bfK=1+\mK$. Let $\bfk$ be the residue field of $\bfK$; it is isomorphic to $\bfF_q$ where $q=p^f$.  According to~\cite[Chap. II.5]{Neukirch}, the multiplicative group $\bfK^\times$ is isomorphic, as an abelian topological group, to the product 
\begin{equation}
\pi^\Z\times \mu_{q-1}\times \Z/p^a\Z\times \Z_p^m
\end{equation}
for some $m\geq 1$ and some $a\geq 0$, with {\emph{both $m$ and $a$ bounded in terms of $[\bfK:\Q_p]$}}. The torsion group $\Z/p^a\Z$ comes from the roots of unity contained in $U_\bfK$. The reduction homomorphism $\VR\to \bfk$ provides a bijection from $\mu_{q-1}=\set{\zeta\in \bfK\; ;\; \zeta^{q-1}=1}$ to $\bfk^\times$. Thus,
\begin{equation}
V(\VR)=F\times (\Z_p^{m})^d,
\end{equation}
where $F$ is the finite group $(\mu_{q-1}\times \Z/p^a\Z)^d$.

The action of $\GL_d(\Z)$ on $V$ preserves $V(\VR)$. This gives an action of the subgroup $\GL_d(\Z)\ltimes  V(\VR)$ of $\Aut(V)$ on $V(\VR)$. A finite index subgroup $G$ of $\GL_d(\Z)\ltimes V(\VR)$ acts trivially 
on the finite part $F=V(\VR)/(\Z_p^m)^d$ and 
by affine transformations on $(\Z_p^m)^d$; more precisely, each element of $G$ acts on $(\Z_p^m)^d$ by
\begin{equation}\label{eq:action_on_Zp_m_d}
(w_1,\ldots, w_d)\mapsto (s_1+\sum_{j=1}^da_{1,j}w_j,\ldots , s_d+\sum_{j=1}^da_{d,j}w_j)
\end{equation}
for some $(a_{i,j})\in \GL_d(\Z)$ and  $(s_j)_{j=1}^d\in (\Z_p^m)^d$. 
 
\begin{eg}\label{eg:finite_orbits_monomial}
Consider the monomial action of $\GL_d(\Z)$ on $\Gm(\overline{\Q_p})^d$. A point has a finite orbit if and only if it is a torsion point, if and only if its coordinates $(\xi_1, \ldots, \xi_d)$ are roots of unity. Thus, finite orbits in  $\Gm(\overline{\Z_p})^d$ are Zariski dense, but only finitely many are contained in $\Gm({\Z_p})^d$. 
\end{eg}

\subsection{Dimension $2$} We keep the same notation, but assume $d=2$. Let $\Gamma$ be a subgroup of $\Aut(V_\VR)$. The intersection $\Gamma_0=\Gamma\cap G$ has a finite index in~$\Gamma$. Let $H_\Gamma$ be the image of $\Gamma$ in $\GL_2(\Z)$, and let $\overline{H}_\Gamma$ be the closure of $H_\Gamma$ in $\GL_2(\Z_p)$; by construction, $\overline{H}_\Gamma$ is contained in $\SL^{\pm}_2(\Z_p)$, the group of matrices with determinant $\pm 1$.

\begin{lem}
If $\dim(V)=2$ and $\Gamma$ is a non-elementary subgroup of $\Aut(V_\VR)$, then $\overline{H}_\Gamma$ is an open subgroup of $\SL_2^\pm(\Z_p)$.
\end{lem}
\begin{proof}
The group $\Gamma$ contains a non-abelian free group, hence so do $\Gamma_0$ and $H_\Gamma$. 
This implies that $\overline{H}_\Gamma$ is open in $\SL_2^\pm(\Z_p)$ (see~\cite[Window 9, Thm. 2]{Lubotzky_Segal:book}).(\footnote{Let us sketch a proof of this last fact. Pick two elements $A$ and $B$ in $\SL_2(\Z_p)$ generating a non-abelian free group. Let $H$ be the closure of $\langle A, B\rangle$ in $\SL_2(\Z_p)$. Taking iterates, we can assume that $A=\Id$ and $B=\Id$ modulo $p^2$. Then, $A^\Z$ and $B^\Z$ are contained in $1$-parameter subgroups parametrized by $\Z_p$. Taking derivatives and using that $\langle A, B\rangle$ is free, one sees that the Lie algebra of $H$ is equal to ${\mathfrak{sl}}_2$; hence $H$ contains an open subgroup of $\SL_2(\Z_p)$. }) 
\end{proof}

So, we assume that $\overline{H}_\Gamma$ is an open subgroup of $\SL_2^\pm(\Z_p)$. There are two cases, in each of them we prove a refined version of Theorem A:

\subsubsection{} Assume that   $\Gamma_0$ is contained in $\SL_2^\pm(\Z_p)$. 
Its action on $(\Z_p^m)^2$, as given in Equation~\eqref{eq:action_on_Zp_m_d}, coincides with its diagonal, linear action on $(\Z_p^2)^m$. 

The action of $\Gamma_0$ on $\Z_p^2$ fixes the origin $o=(0,0)$, and if $u\in \Z_p^2\setminus \set{o}$ its orbit closure is  open. So, if $m=1$, $\Gamma_0$ has one fixed point and infinitely many open orbit closures. 
If $m\geq 2$, the function $\det(v_1,v_2)$ is invariant, hence no orbit closure contains an open set. (If $m\geq 3$, any $m$-tuple of vectors $(v_i)$ satisfies linear relations of length $3$, such as $b_1v_1+b_2v_2+b_3v_3=0$ for some non-zero vector $(b_1,b_2,b_3)\in \Z_p^3$. This shows also that no orbit closure is open). Thus, one sees that
\begin{enumerate}[\rm (1)]
\item the general orbit closure of $\Gamma$ on $V(\VR)$ is open if and only if $m=1$, that is, if and only if $\bfK=\Q_p$; 
\item $\Gamma$ has a finite number of finite orbits (corresponding to the elements of $F\times\set{o}$); 
\item there are infinitely many distinct orbit closures.
\end{enumerate}

\subsubsection{} Suppose that $\Gamma_0$ is not conjugate to a subgroup of $\SL_2^\pm(\Z_p)$. Since $\overline{H}_\Gamma$ contains an open set, the closure of $\Gamma_0$ in the affine group $\SL_2^\pm(\Z_p)\ltimes \Z_p^2$ contains an open neighborhood of the origin in $\Z_p^2$.  This implies that every orbit closure is open when $m = 1$ or $2$. On the other hand, when $m\geq 3$, the interior of every orbit closure is empty because the function $\det(v_2-v_1, v_3-v_1)$ is invariant.

\subsection{An example: the Cayley cubic}\label{par:cayley_cubic} Consider the surface $V=\Gm\times \Gm$ with the monomial action of $\GL_2(\Z)$: each element $M$ of $\GL_2(\Z)$ determines an automorphism $f_M$ of $V$. The involution $\eta:=f_{-\Id}$ is a central element, and the quotient $V/\langle \eta\rangle$ is the Cayley cubic. It is an affine cubic with four isolated singularities (the maximal number for a cubic). More precisely, the map $(v_1,v_2)\in V\mapsto -(v_1+1/v_1, v_2+1/v_2, v_1v_2+1/(v_1v_2))\in \A^3$ identifies $V/\langle \eta\rangle$ to the affine surface of equation $\vx^2+\vy^2+\vz^2+\vx\vy\vz=4$ (it is a member of the Markov family described below).
Since $-\Id$ is central in $\GL_2(\Z)$, $\PGL_2(\Z)$ acts by automorphisms on $V/\langle \eta\rangle$. On $V$ and $V/\langle \eta\rangle$, $f_M$ is loxodromic if and only $\Tr(M)^2>4$, and $f_M$ is a Jonquières twist if and only if $M$ has infinite order and $\Tr(M)^2=4$. If $f_M$ is elliptic, then $f_M$ has finite order. In particular, there are Jonquières twists on singular affine surfaces, a phenomenon that appears on every singular Markov surface (see \S\ref{par:examples_markov}). 
 
\section{Decomposition into orbit closures}\label{par:proof_of_thmA_part1}

In this section we prove Theorem~A, except for the statement concerning stationary measures. So, let $X$ be an affine surface defined over $\Z_p$, and let $\Gamma$ 
be a subgroup of $\Aut(X_{\Z_p})$ such that 
\begin{enumerate}[(a)]
\item $\Gamma$ contains a loxodromic element $g$;
\item $\Gamma$ contains a non loxodromic element $f$ of infinite order.
\end{enumerate}

From the results of Section~\ref{par:case_of_GmxGm}, we can assume that $X$ is not isomorphic to $\Gm\times \Gm$ over $\bfKbar$. 
Thus, by Corollary~\ref{cor:application2_of_thm_C}, if $f$ is parabolic, it is a Jonquières twist acting by a finite order automorphism on the base of its invariant fibration $\pi\colon X\to B$. 
Changing $f$ into a positive iterate, we assume that it preserves each fiber of $\pi$, i.e.\ $f_B=\Id_B$.
 
 We denote by $\Gamma_0$ a normal, finite index subgroup of $\Gamma$ that satisfies the conclusion of Theorem~\ref{thm:good_finite_cover}. Then, $\Gamma_0$ is torsion free (see Remark~\ref{rem:no_torsion}).
 As in Section~\ref{par:lie_dimension}, we denote by $s_\Gamma(x)$ 
 the dimension of $L_\Gamma(x)$ for $x\in X(\Z_p)$. 
 
 \medskip
  
\noindent{\bf{Step 1}}.-- {\emph{There are only finitely many finite orbits in $X(\Z_p)$.}} 

\smallskip

This finiteness result holds over any $p$-adic local field:

\medskip

\begin{thm-D}
Let $\bfK$ be a $p$-adic local field, with valuation ring $\VR$. Let $X$ be an affine surface defined over $\VR$, and let $\Gamma$ be a non-elementary subgroup of $\Aut(X_\VR)$ containing a non-loxodromic element of infinite order. Then, $\Gamma$ has at most finitely many finite orbits in 
$X(\VR)$. \end{thm-D}

\begin{proof} Let $\Per_\Gamma(\VR) \subset X(\VR)$ be the set of points with a finite $\Gamma$-orbit. 
Such a point is periodic under the action of $f$ and all its conjugates $hfh^{-1}$, for $h\in \Gamma$. 
Thus, 
\begin{equation}
\Per_\Gamma(\VR)\subset \bigcap_{h\in \Gamma} P_{hfh^{-1}}(\VR)
\end{equation}
where $P_{hfh^{-1}}=h(P_f)$ and $P_f$
is defined in Lemma~\ref{lem:finite_orbits_of_elliptic} when $f$ is elliptic or in Theorem~\ref{thm:parabolic_periodic_algebraic} when $f$ is parabolic. Thus, $ \cap_{h\in \Gamma} P_{hfh^{-1}}$ is an intersection of Zariski closed subsets of dimension $\leq 1$, and, as such, it 
is Zariski closed of dimension $\leq 1$. Since it is $\Gamma$-invariant, it must be finite, because the loxodromic automorphism element $g$ does not preserve any curve (Theorem~\ref{thm:Abboud_invariant_curve}). Thus, 
$\Per_\Gamma(\VR)$ is finite and coincides with the set of $\VR$-points in  $ \cap_{h\in \Gamma} P_{hfh^{-1}}$. \end{proof}


\noindent{\bf{Step 2}}.-- {\emph{For $x\in X(\Z_p)$, we have $s_\Gamma(x)=2$ if the orbit of $x$ is infinite and $s_\Gamma(x)=0$ otherwise.}}

\smallskip

Indeed, consider a point $x$ of $X(\Z_p)$ and a neighborhood $\U(x)\simeq \Z_p^2$ of $x$, as well as the subgroup $\Gamma_0\subset\Gamma$, as in Theorem~\ref{thm:good_finite_cover}. If the orbit of $x$ is finite, it is fixed by the flow $\Phi_h\colon \Z_p\times\U(x)\to \U(x)$
associated with any $h\in \Gamma_0$. Hence, $s_{\Gamma_0}(x)=s_\Gamma(x)=0$. Conversely, if $s_\Gamma(x)=0$, then the vector field $\Theta_h$ vanishes at $x$ for every $h\in \Gamma_0$, which implies that $h(x)=\Phi_h^1(x)=x$. Therefore, $\Gamma_0$ fixes
$x$ and the $\Gamma$-orbit has cardinality $\vert\Gamma(x)\vert\leq [\Gamma : \Gamma_0]$.

Now, consider the set $S=\set{y\in X(\Z_p)\; ;\; s_\Gamma(y)\leq 1}$. Let us prove that $S$ is the set of $\Z_p$-points of a Zariski closed, $\Gamma$-invariant set. For this, we consider the distribution $L_{f}$: by definition, if $f$ is elliptic, then $L_f$ is the distribution of lines determined by 
an algebraic vector field $a_f$ (see Section~\ref{par:bell_poonen_elliptic}), and if $f$ is parabolic, $L_f$ is the distribution of lines tangent to the $f$-invariant fibration $\pi\colon X\to B$ (see Remark~\ref{rem:L_for_parabolic}). Thus, given any $h\in \Gamma$, 
$L_{hfh^{-1}}=h_*L_f$ is a globally defined algebraic distribution of lines in the tangent space $TX$. This implies that the tangency 
locus 
\begin{equation}
T_\Gamma=\set{ y\in X\; ;\; \dim {\mathrm{Vect}} (L_{hfh^{-1}}(x); h\in \Gamma)\leq 1 }
\end{equation}
is a $\Gamma$-invariant algebraic subset (\footnote{Indeed, the tangency locus between two algebraic distributions of lines is an algebraic subset of $X$, and an intersection of algebraic subsets is algebraic. 
}). 
Since $\Gamma$ contains loxodromic elements, $T_\Gamma$ coincides with $X$ or is a finite set. 

\begin{lem}
The algebraic set $T_\Gamma$ is a finite subset of $X$.
\end{lem}

\begin{proof} We must show that $T_\Gamma$ does not coincide with $X$. Otherwise, the distribution of lines $L_f$ is invariant under the action of $\Gamma$. This invariant line field is not tangent to a fibration, because $\Gamma$ contains loxodromic elements, and such automorphisms do not preserve any fibration. 
Thus, $f$ is elliptic, and 
\begin{enumerate}[-- ]
\item $L_f$ is determined by an algebraic vector field $a_f$ on $X$, as in Section~\ref{par:bell_poonen_elliptic}; 
\item the algebraic group $A$ in which $f^\Z$ is Zariski dense has dimension $2$, since otherwise its orbits would determine an invariant fibration tangent to $L_f$, but $g$ does not preserve any fibration;
\item  $A^\circ$ is isomorphic to $\Ga\times \Gm$ or to $\Gm\times \Gm$, has an open orbit, and  almost all orbits of $f$ are Zariski dense in $X$.
\end{enumerate}
We shall give two arguments to reach a contradiction. The first one is based on the classification of birational symmetries of foliations; the second one is a simple variation on Theorem~\ref{thm:GmGm_and_invariant_analytic_curve}.

\smallskip
 
{\emph{First argument}}.-- Let $\overline{X}$ 
be a completion of $X$, smooth at infinity.  Since $\Gamma$ contains loxodromic elements and the pair
$(X,\Gamma)$ is defined over a field of characteristic $0$, we can apply~\cite[Cor.\ 1.3]{Cantat-Favre} to 
the triple $(\overline{X}, \Gamma, L_f)$. This gives four possibilities. In the first case $\overline{X}$ is 
birationally equivalent to an abelian surface $A$. Since the group $\Bir(A)$ coincides with $\Aut(A)$, Proposition~\ref{pro:not_conjugate_automorphism_projective} provides a contradiction. The same argument applies when $\overline{X}$ is birationally equivalent to the quotient of such an abelian surface by a finite group action. The third case leads to $X=\Gm\times \Gm$, which is excluded by hypothesis. And the fourth one leads to $X$ being the quotient of $ \Gm\times \Gm$ by $(u,v)\mapsto (1/u,1/v)$: this is excluded too, because this surfaces does not have infinite order elliptic elements. 

\smallskip
 
{\emph{Second argument}}.-- If $g$ is an element of $\Gamma$, then $g_*a_f$ is a new vector field and is proportional to $a_f$. This means that there is a rational function $\xi_g$ on $X$ such that $g_*a_f=\xi_g\cdot a_f$. This rational function may have poles on $\{a_f=0\}$, but according to Proposition \ref{pro:Pf_is_finite_if_dimA_is_2}, this set is finite; thus, $\xi_g$ is actually a regular function on $X$. Moreover, $\xi_g$ cannot vanish because $g$ is an automorphism. Thus, by Theorem~\ref{thm:Iitaka-Abboud}, $\xi_g$ is a constant. This implies that $\Gamma$ normalizes the algebraic group $A^\circ$ associated with the elliptic element $f$. The automorphisms of $\Ga\times \Gm$ preserve the fibration onto $\Gm$, hence if $A^\circ \simeq \Ga\times \Gm$, there would be a $\Gamma$-invariant fibration, a contradiction. Thus $A^\circ \simeq \Gm\times \Gm$, and as in the proof of Theorem~\ref{thm:GmGm_and_invariant_analytic_curve}, we see that 
$X$ is isomorphic to $\Gm\times \Gm$ over $\bfKbar$, a contradiction.
\end{proof}

 By noetherianity, there is a finite number of conjugates $f_i=h_ifh_i^{-1}$, $1\leq i\leq q$, such that $T_\Gamma$ 
coincides with $\set{ y\in X\; ;\; \dim {\mathrm{Vect}} (L_{f_i}(y); 1\leq i\leq q)\leq 1 }$.
Let $x$ be a point with an infinite orbit. 
As $T_\Gamma$ is finite, one can find a point $x'$ in the orbit of $x$ that is not on $T_\Gamma$. For such a point, $\Theta_{f_i}(x')$ generate $T_{x'}X$, and thus $s_\Gamma(x')=2$.
But the set $\set{s_\Gamma=2}$ is $\Gamma$-invariant, thus $s_\Gamma(x)=2$. This concludes Step~2.

 \medskip

\noindent{\bf{Step 3}}.-- To conclude, we apply Theorem~B  and Remark~\ref{rem:orbit_closure} from Section~\ref{par:general_finiteness_result}. These results show that in the complement of the finite set $\Per_\Gamma(\Z_p)$, every orbit closure is open. Being open, the orbit closures intersect if and only if they coincide, so they form a partition of $X(\VR)$, as desired. 

%
%

\section{Stationary measures}

The goal of this section is to prove Proposition~\ref{pro:stationary_measure_Isom}. This proposition completes the proof of Theorem~A.

\subsection{Compact groups}\label{par:compact_groups}

Let $G$ be a compact topological group, and let $\omega_G$ be its Haar measure, normalized so that $\omega_G(G)=1$. Let $S$ be a closed subgroup of $G$, and  let $q\colon G\to G/S$ be the quotient map. The group $G$ acts by left translations on $G/S$, and the push-forward measure $\omega_{G/S}=q_*\omega_G$ is a $G$-invariant probability measure on $G/S$. Now, let $\mu$ be a probability measure on $G$ whose support generates a dense subgroup of $G$. If $\nu$ is a $\mu$-stationary measure on $G/S$, then it is automatically $G$-invariant (this follows from the maximum principle for continuous functions on $G$, see also~\cite[Thm. 3.5]{furstenberg:stiffness}), and the uniqueness of the Haar measure implies that $\nu$ coincides with $\omega_{G/S}$.

\begin{rem} Let $\mu$ be a probability measure on $G$ such that (i) the subgroup generated by the support $\Supp(\mu)$ is dense in $G$, and (ii) $\Supp(\mu)$ is not contained in a coset of a closed normal subgroup of $G$. Then, the convolutions $\mu^{\star n}$ converge towards $\omega_G$ as $n$ goes to $+\infty$ (see~\cite[Main Thm. 3.3.5]{stromberg:1960}). Both (i) and (ii) are necessary, but (ii) can be dropped if one is interested only in the classification of stationary measures. \end{rem}

\subsection{Isometry groups}  Let $\bfK$ be a non-archimedian local field. 
As above, we endow the affine space $\A^N_\bfK$ with the distance given by the sup-norm: 
\begin{equation}
\dist(x,y)=\max_{i=1, \ldots, N}\av{x_i-y_i}.
\end{equation}
Let $V\subset \A^N$ be a subvariety defined over $\VR$. If $f$ is an element of $\Aut(V_\VR)$, then 
$f$ is the restriction to $V$ of an endomorphim $\tilde{f}$ of $\A^N$ defined over $\VR$. As explained in~\cite[\S{2.1.2}]{CantatXie2018}, $f$ induces a $1$-Lipschitz homeomorphism of $V(\VR)$ with respect to $\dist(\cdot, \cdot)$. Since $f^{-1}$ is also $1$-Lipschitz, $f$ acts by isometry on $V(\VR)$. This provides a
homomorphism 
\begin{equation}
\Aut(V_\VR)\to \Isom(V(\VR)).
\end{equation}
Moreover, with the topology induced by uniform convergence, $\Isom(V(\VR))$ is a compact group. 
 
\begin{pro}\label{pro:stationary_measure_Isom}
Let $\bfK$ be a non-archimedian local field. Let $V$ be an affine variety defined over $\VR$.
Let $\Gamma$ be a subgroup of $\Aut(V_\VR)$ and let $G$ be the closure of $\Gamma$ in $\Isom(V(\VR))$. Let $\mu$ be a probability measure on $G$ whose support generates a dense subgroup of $G$. Then, every orbit closure $\overline{\Gamma(x)}\subset V(\VR)$ supports a unique $\mu$-stationary probability measure. This measure is the unique $G$-invariant probability measure on $\overline{\Gamma(x)}$.
\end{pro}
\begin{proof}
The orbit closure $\overline{\Gamma(x)}$ is equal to $G(x)$ and can be identified (as a topological $G$-space) to $G/S_x$, where $S_x\subset G$ is the stabilizer of $x$ in $G$. Let $\omega_x$ be the push-forward of $\omega_G$ on $G(x)$ by the map $g\mapsto g(x)$. Then, the result follows from Section~\ref{par:compact_groups}.
\end{proof}

As an example, consider a countable group $\Gamma \subset\Aut(V_\VR)$ and a probability measure $\mu=\sum_{f\in \Gamma}\mu(f)\delta_f$ such that $\set{ f\in \Gamma\; ; \; \mu(f)\neq 0}$ generates $\Gamma$(\footnote{The reason for choosing $\mu$ as a probability measure on $G$ (rather than on $\Gamma$) in Proposition~\ref{pro:stationary_measure_Isom} is because otherwise we should introduce a $\sigma$-algebra on $\Aut(V_\VR)$ for which the map $f\in \Aut(V_\VR)\mapsto f_{\vert V(\VR)}\in \Isom(V(\VR))$ is measurable.}).  In the example of Markov surfaces, $\mu$ will be symmetric, with finite support.

\section{The Markov surfaces}\label{par:markov_surfaces}

In this section, we study the family of surfaces $S_{A,B,C,D}\subset \A^3$ defined by the equation 
\begin{equation}
\vx^2+\vy^2+\vz^2+\vx \vy \vz = A\vx + B\vy +C\vz+D.
\end{equation}
Here, we suppose that the parameters $A$, $B$, $C$, and $D$ are in the valuation ring $\VR$ of a $p$-adic local field $\bfK$. For simplicity, we denote $S_{A,B,C,D}$ by $S$. 
The group $\Gamma$ that we shall study is the one generated by the Vieta involutions $s_1(\vx, \vy, \vz)=(-\vx+A-\vy\vz, \vy, \vz)$, $s_2(\vx, \vy, \vz)=(\vx, - \vy+B-\vz\vx, \vz)$, and $s_3(\vx, \vy, \vz)=(\vx, \vy, -\vz+C-\vx\vy)$. It is a finite index subgroup of $\Aut(S)$, defined over the ring $\Z[A,B,C,D]$, and hence on $\VR$. 

The surface $S$, if smooth, supports a regular $2$-form $\Omega$, defined locally by 
\begin{equation}
\Omega = \frac{d\vx\wedge d\vy}{2\vz-C+\vx\vy}=\frac{d\vy\wedge d\vz}{2\vx-A+\vy\vz}=\frac{d\vz\wedge d\vx}{2\vy-B+\vz\vx}.
\end{equation}
By $p$-adic integration, this gives a measure on $S(\VR)$ (see Igusa's book~\cite[Chap. 7.4]{igusa}, or~\cite{popa:notes}). If $U$ is a non-empty open subset of $S(\VR)$ (with respect to the $p$-adic topology), we shall refer to the measure 
\begin{equation}
\frac{1}{\int_U\Omega }\Omega_{\vert U}
\end{equation}
as the {\bf{normalized symplectic measure}} on $U$. 

\subsection{Parabolic automorphisms and bounded orbits}  The composition $f=s_1\circ s_2$ is an automorphism that preserves the fibration defined by $\pi_\vz\colon (\vx, \vy, \vz)\mapsto \vz$. If we cut $S$ by a horizontal 
plane $\set{\vz=z}$, for some $z\in \bfK$, we get a conic $C_z$, with the equation 
\begin{equation}
\vx^2+\vy^2+z \vx \vy  = A\vx + B\vy +(Cz+D-z^2).
\end{equation}
The action of $f$ on this conic section is given by the restriction of an affine map $A_z(\vx,\vy)=M_z(\vx,\vy)+T_z$ where 
\begin{equation}
M_z=\left( \begin{array}{cc} z^2-1 & z \\ -z & -1 \end{array}\right) \quad {\text{and}} \quad T_z=\left( \begin{array}{c} A-Bz  \\ B \end{array}\right).
\end{equation}
The determinant of $M_z$ is $1$, its trace is $z^2-2$, so its eigenvalues are $\alpha$ and $1/\alpha$ 
with $\alpha+1/\alpha = z^2-2$. Extracting a root, we get $2\alpha(z)=(z^2-2)\pm z\sqrt{z^2-4}$. Thus, 
over a quadratic extension of $\bfK$, $A_z$ is conjugate to the linear map $(\vx,\vy)\mapsto (\alpha(z)\vx, \vy/\alpha(z))$. This implies that $\deg(f^n)$ grows like $2n$ as $n$ goes to $+\infty$, so $f$ is a Jonquières twist. Also, if $\av{z}>1$, we see that one of the two eigenvalues, say $\alpha(z)$, satisfies $\av{\alpha(z)}>1$.

 The next lemma follows from a direct computation. 

\begin{lem}
If $z^2\neq 4$, the affine map $A_z$ has a fixed point in $\bfK^2$; this fixed point is on $C_z$ (equivalently on $S$) if and only if $z=0$, in which case $A_z\circ A_z=\Id$.

If $z=2\epsilon$, with $\epsilon = \pm 1$, then $A_z$ has either no fixed point in $\bfK^2$, or a line of fixed point defined by the equation $\epsilon x=y+B/2$; this last case occurs if and only if $\epsilon A+B=0$.
\end{lem}

Thus, if $(x,y,z)$ is a point of $S(\bfK)$ with $\av{z}>1$, the eigenvalues of $M_z$ satisfy 
$\av{\alpha(z)}>1>\av{1/\alpha(z)}$, and the unique fixed point of $A_f$ in $\bfK$ is outside $C_z$. This implies that the orbit of $(x,y,z)$ under the action of $f^\Z$ is not bounded. 

\begin{pro}
A point $(x,y,z)\in S(\bfK)$ has a bounded $\Gamma$-orbit if and only if $(x,y,z)\in S(\VR)$. 
Thus, $S(\VR)$ is the unique maximal, compact, invariant subset of $S(\bfK)$.
\end{pro}

\begin{proof}
The compact set $S(\VR)$ is $\Gamma$-invariant. If $(x,y,z)$ is not in $S(\VR)$, then one of the coordinates has an absolute value $>1$. If $\av{z}>1$, then we know that $f^\Z(x,y,z)$ is not bounded. 
Otherwise, we may apply $s_3\circ s_1$ or $s_2\circ s_3$ to conclude in the same way. 
\end{proof}

\subsection{General parameters}\label{par:markov_general_parameter} The finite orbits of $\Gamma$ in $S_{A,B,C,D}(\bfKbar)$ have been classified in~\cite{Lisovyy-Tykhyy} (see also \cite[\S{3.1}]{cantat-dupont-martinbaillon}). For a general parameter $(A,B,C,D)$, every $\Gamma$-orbit in $S=S_{A,B,C,D}$
is infinite. 

\smallskip

\begin{thm-E}
Assume that $A$, $B$, $C$, and $D$ are in $\Z_p$, and that $S(\Z_p)$ does not contain any finite orbit. 
Let $\mu$ be a probability measure on $\set{s_1,s_2,s_3}$ with $\mu(s_1)\mu(s_2)\mu(s_3)>0$. 
Let $\nu$ be an ergodic $\mu$-stationary probability measure on $S(\Q_p)$. Then, 
\begin{enumerate}[\rm(1)]
\item $\nu$ is supported on an orbit closure $\overline{\Gamma(x)}$, for some $x\in S(\Z_p)$;
\item this orbit closure is a clopen subset;
\item $\nu$ is equal to the normalized symplectic measure on $\overline{\Gamma(x)}$.
\end{enumerate}
Moreover, there are only finitely many orbit closures in $S(\Z_p)$, and the convex set of $\mu$-stationary measures on $S(\Q_p)$ is a finite dimensional simplex.
\end{thm-E}


\begin{proof}
According to~\cite[Pro. 6.1, Rem. 6.2]{cantat-dupont-martinbaillon}, $\nu$ is supported on a compact invariant subset. 
Thus, it is supported in $S(\VR)$, and by Theorem~A, ergodicity, and Proposition~\ref{pro:stationary_measure_Isom}, it is the unique invariant measure on an orbit closure $\overline{\Gamma(x)}$. But $\Omega_{\overline{\Gamma(x)}}$ is such an invariant measure, thus $\nu=\Omega_{\overline{\Gamma(x)}}$.
\end{proof}

\subsection{Examples}\label{par:examples_markov} Assume that $S(\VR)$ contains a finite orbit $F$, and that $F$ is not an isolated subset in $S(\VR)$. Since the action of $\Gamma$ is isometric, there are infinitely orbit closures accumulating to $F$, and this implies that $S(\VR)$ supports infinitely many ergodic stationary measures. 

For instance, the origin $o=(0,0,0)$ is a singularity of $S_o=S_{0,0,0,0}$, it is not isolated in $S_o(\Z_p)$, and it is fixed by $\Gamma$. On the other hand, if we focus on the subset $S'=\set{(x,y,z)\in S_o\; ; \; \dist(o, (x,y,z))=1}$, then the work of Bourgain, Gamburd and Sarnak and of Chen shows that, for $p$ large, the group generated by $\Gamma$ and the permutations of coordinates act topologically transitively on $S'$: 
every orbit in $S'$ is dense in $S'$ (see~\cite{Bourgain-Gamburd-Sarnak:Markov, Chen:Markoff}). Thus, $S'$ supports a unique stationary measure.

\section{An open question} 

\subsection{Simple examples} In the following examples, $f$ and $g$ are automorphisms of $\A^2$, both defined over $\Z$, and $\Gamma=\langle f, g\rangle$. In the first set of examples, $\Gamma$ has only finitely many orbits in $\A^2(\Z)$, but the third example has infinitely many orbits, and the orbits of $\Gamma$ in $\A^2(\Z)$ for the example in \S~\ref{par:henon_pure} are ``sparsed'' in terms of height. In each case, we are interested in the number of orbits modulo $p$ as $p$ increases.

\subsubsection{} This first example can be handled directly, without any reference to Theorem~A. Let $t$ be a positive integer. Let $Q(\vx)$ be a polynomial function of degree $d\geq 2$ with coefficients in $\Z$. Then, we take
 $g(\vx, \vy)=(\vy+Q(\vx), \vx)$  and $f(\vx,\vy)=(\vx+t,\vy)$. 
Suppose that $p$ does not divide $t$. Let $z=(x,y)$ be an element of $\A^2(\Z_p)$ and let $U$ be the closure of the $\Gamma$-orbit of $z$. Since $t\Z$ is dense in $\Z_p$, $U$ is the preimage of a subset $U_0$ of 
$\A^1(\Z_p)$ under the second projection $\pi(\vx,\vy)=\vy$. On the other hand, 
$g\circ f\circ g^{-1}(\vx,\vy)=(x+Q(\vy+t)-Q(\vy), \vy+t)$ 
is in $\Gamma$. Thus $\pi(U)$ is invariant under the translation $\vy\mapsto \vy+t$. This implies that $U=\A^2(\Z_p)$, so every $\Gamma$-orbit is dense in $\A^2(\Z_p)$. 
On the other hand, if $t$ and $Q$ are divisible by $p$, then the action of $\Gamma$ on $\A^2(\Z/p\Z)$ has at least $p(p+1)/2$ orbits, and no orbit is dense in $\A^2(\Z_p)$. 
This argument also shows that when $t=\pm 1$, then $\Gamma$ acts transtively on $\A^2(\Z)$.

\subsubsection{} Consider a translation $f(\vx,\vy)=(\vx+t,\vy)$, as in the previous example, and any loxodromic 
automorphism $g\in \Aut(\A^2_\Z)$.  If $p$ does not divide $t$, every orbit of $\Gamma$ modulo $p$ is the preimage of a subset of $\A^1(\Z/p\Z)$ under the second projection. The automorphism $h=g\circ f\circ g^{-1}$ acts by translations along the fibers of the function $g_2:=g^*(\vy)$ and every orbit of $h$ contains $p$ points. Let $d$ be the degree of $g_2$ with respect to the $\vy$ variable. Then, one can easily verify that $\Gamma$ has at most $d$ orbits modulo $p$.




\subsubsection{} Another interesting case is when $f\in \GL_2(\Z)$ is a linear diagonalizable automorphism and $g\in \Aut(\A^2_\Z)$ is any loxodromic automorphism. For instance, one can take $f(\vx,\vy)=(2\vx+\vy, \vx+\vy)$. Numerical simulations suggest again that the number of orbits of $\Gamma$ modulo $p$ is uniformly bounded, independently of $p$. It would be interesting to adapt the techniques of~\cite{Bourgain-Gamburd-Sarnak:Markov} to this example. 

\subsubsection{}\label{par:henon_pure}
\iftrue
    Let $h_0$ and $g$ be the elements of $\Aut(\A^2_\Z)$ defined by
    \begin{align*}
        h_0(\vx,\vy) & = (2\vx+\vy,\vx+\vy), \\
        g(\vx,\vy) &= (\vy+\vx^2+5,-\vx).
    \end{align*}
    Let $p<10^4$ be a prime. Then, our numerical simulations show that 
    if $p\neq 5$, the group $\langle g, h_0\circ g\circ h_0^{-1}\rangle$ acts transitively on $\A^2(\Z/p\Z)$; 
    if $p=5$, the origin is fixed and the group acts transitively on $\A^2(\Z/p\Z)\setminus\{(0,0\}$.
    On the other hand, each orbit of $g$ alone contains at most $c(g) p\log(p)$ points, for $p$ in this range, with $c(g)=2$ (for other Hénon maps, $c(g)$ can be larger).
    We have no explanation yet for these phenomena, and we do not know whether the range $2\leq p \leq 10^4$ is big enough to derive precise conjectures. 

\begin{figure}[h]\label{fig:3Henon}
\includegraphics[width=1.0\textwidth]{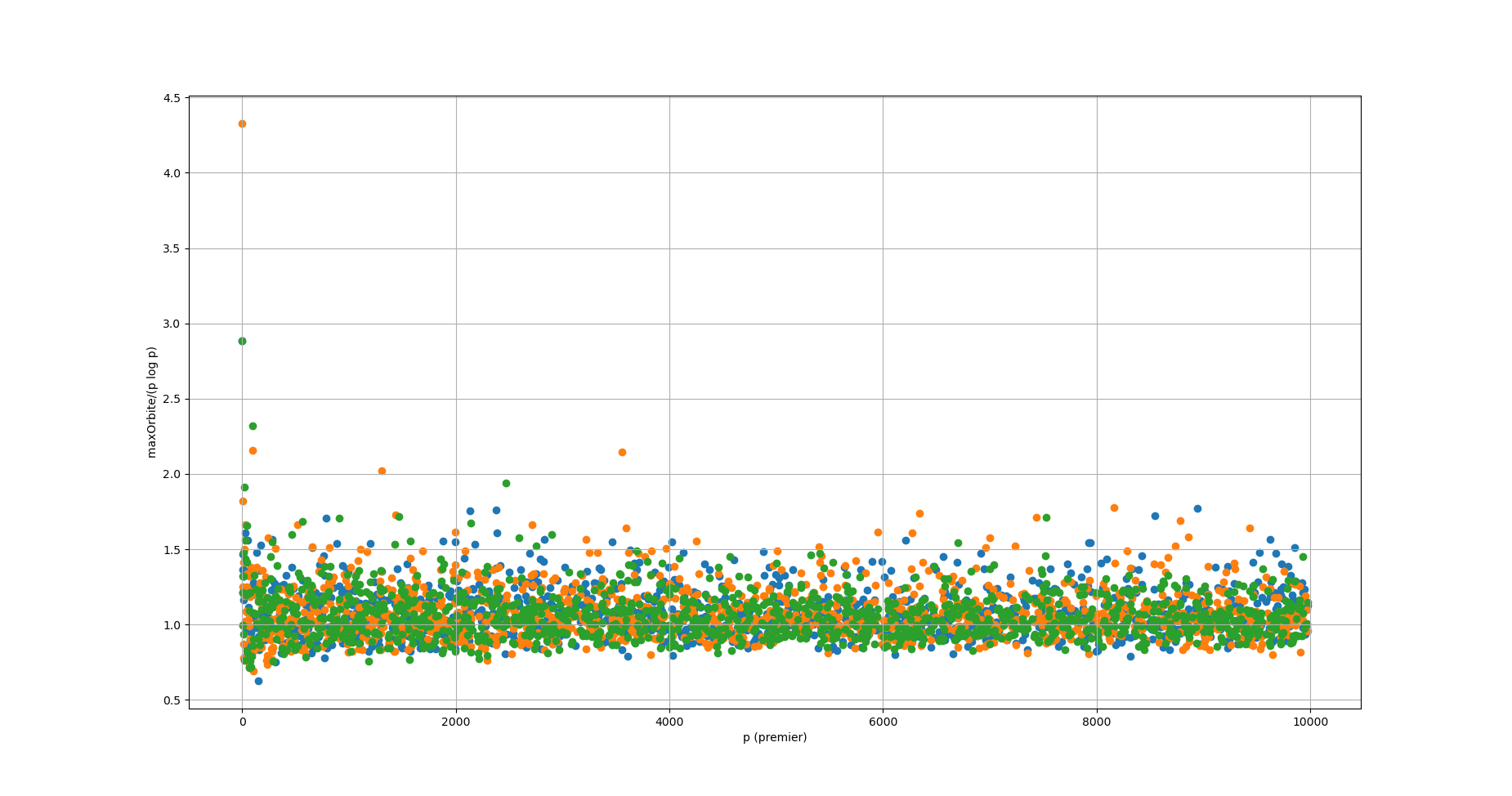}
\caption{This picture represents, for a given $p$, the maximal length of an orbit of $g$ modulo $p$ divided by $p\log(p)$. Blue  points correspond to $g_1(\vx,\vy)=(\vy+\vx^2+5,-\vx)$, orange points to $g_2(\vx,\vy)=(-\vy, \vx+\vy^3+2)$, and green points to $g_3=g_2\circ h_0\circ g_1$ where $h_0(\vx,\vy) = (2\vx+\vy,\vx+\vy)$ is as above. } 
\end{figure}
    
\else
    Let $f$ and $g$ be the elements of $\Aut(\A^2_\Z)$ defined by
    \begin{align*}
        f(\vx,\vy) & = (2\vx+\vy,\vx+\vy), \\
        g(\vx,\vy) &= (\vy+\vx^2+5,-\vx).
    \end{align*}
    Let $p<10^4$ be a prime. Then, our numerical simulations show that 
    if $p\neq 5$, the group $\langle g, f\circ g\circ f^{-1}\rangle$ acts transitively on $\A^2(\Z/p\Z)$; 
    if $p=5$, the origin is fixed and the group acts transitively on $\A^2(\Z/p\Z)\setminus\{(0,0\}$.
\fi

%
%
\subsection{Uniform bounds}

Though our numerical simulations concern only small primes (say $p <10^4$), 
the examples from the previous section suggest the following question. 

\begin{que}
{\sl{Let $\Gamma$ be a non-elementary subgroup of $\Aut(\A^2_\Z)$. Is the number of orbit closures of $\Gamma$ in $\A^2(\Z_p)$ bounded by a uniform constant $b(\Gamma)$ that does not depend on $p$?}}
\end{que}

This question is already interesting when $\Gamma$ contains an element $f$ of infinite order that is not loxodromic (as in Theorem~A). 
This problem is also related to the following one: {\sl{under which conditions can we ensure that transitivity modulo $p$ (resp. some $p^k$) implies topological transitivity over $\Z_p$ (i.e. modulo $p^l$ for all $l\geq 1$)?}} 



\appendix
\section{A proof of the Bell-Poonen Theorem}

In this appendix, we present an alternative proof of the Bell-Poonen theorem (see Theorem \ref{thm:bell_poonen} above); we refer to~\cite{Bell, Bell:corrigendum, Poonen2014} for the orginial proofs. The method naturally provides an explicit formula for the vector field $\Theta_f$ defined in Equation~\eqref{eqn:champ-vectoriel-definition}: see Equation~\eqref{eqn:formule-de-champ-vectoriel} below.

\subsection{Functional Analytic Ingredients}
We use the notations $R=\VR$, $\U=R^m$, and $R\langle\vx\rangle$ from Section \ref{par:analytic_diffeomorphism_polydisk}. On $\U$, we use the $\ell^\infty$-norm $|(x_1,\ldots,x_m)|=\max_{1\leq i\leq m}|x_i|$, and on the ring $R\langle\vx\rangle$, we use the Gauss norm $\parallel{}\cdot{}\parallel$ 
defined by
\begin{equation}
\left\|g\right\|=\sup_I|a_I|
\end{equation}
for every $g=\sum_Ia_I\vx^I\in R\langle\vx\rangle$. 
If we view $g\in R\langle\vx\rangle$ as a Tate analytic function $g\colon\U\to R$, the norm is an \emph{upper bound} of the function values: i.e., we have $|g(z)|\leq\|g\|$ for all $z\in\U$ \cite[Prop. 5.1.4/2]{bosch-guntzer-remmert}. The function values may not attain the Gauss norm: see Remark \ref{rem:ceci-nest-pas-une-norme} above. 
Note that our definition of $f\equiv\Id\pmod{p^c}$ in \S{\ref{par:reduction_and_bell_poonen}} asserts $\|f-\Id\|\leq p^{-c}$.

By this norm, $R\langle\vx\rangle$ is a complete (ultra)metric space \cite[Prop. 1.4.1/3]{bosch-guntzer-remmert}. Although $R$ is not a field, we extend the notion of (faithfully) normed $R$-modules from \cite[Def. 2.1.1/1]{bosch-guntzer-remmert} and say that $R\langle\vx\rangle$ is an $R$-Banach space with respect to the Gauss norm $\parallel{}\cdot{}\parallel$. 
We define the norm on $R\langle\vx\rangle^m$ by the supremum of the Gauss norms on each component.

%
%
%
%

\subsection{Construction of the Flow}
Let $A$ denote the $R$-algebra of bounded linear endomorphisms $R\langle\vx\rangle^m\to R\langle\vx\rangle^m$, with norm given by the operator norm $\|{}\cdot{}\|_{op}$. This algebra is \emph{normal} in the sense that within the extension of scalars $A\otimes_R\bfK$, the ball of radius $1$ centered at the origin is precisely $A$.

Let $f\colon\U\to\U$ be a Tate analytic map such that $f\equiv\Id\pmod{p^c}$ for some real number $c>1/(p-1)$. By precomposition, we obtain an operator 
\begin{equation}
T_f\colon R\langle\vx\rangle^m\to R\langle\vx\rangle^m, \quad g\mapsto g\circ f. 
\end{equation}
By \cite[Lem. 2.1(3)]{CantatXie2018}, denoting by $\Idop$ the identity operator on $R\langle\vx\rangle^m$, 
we have 
\begin{equation}
\|T_f-\Idop\|_{op} \leq \|f-\Id \| \leq p^{-c}.
\end{equation}

In what follows, we denote by $\exp$ and $\log$ the usual (formal) power series 
\begin{align}
\exp(\mathbf{a}) &= \sum_{k=0}^\infty\frac{\mathbf{a}^k}{k!}, \label{eqn:exp-series}\\
\log(1+\mathbf{a}) &= \sum_{k=1}^\infty\frac{(-1)^{k-1}}{k}\mathbf{a}^k. \label{eqn:log-series}
\end{align}

The proof given in Chapter IV of  \cite{koblitz:p-adic}, and in particular the Proposition on page 81 of this book, generalizes to normal $R$-Banach algebras (with unit denoted by $1$, or by $\Idop$ in the case of endomorphisms of $R\langle\vx\rangle^m$). This establishes the following properties.
\begin{lem}
Let $A$ be a normal $R$-Banach algebra. 
\begin{enumerate}[\rm (1)]
\item The  series $\exp(\mathbf{a})$ converges on the ball $\{a\in A \; ; \;  \|a\|<p^{-1/(p-1)} \}$. 
\item The series $\log(1+\mathbf{a})$ converges on the open unit ball  $\{a\in A \; ; \;  \|a\| < 1 \}$.
\item The exponential function $\exp(\mathbf{a})$ defines a $1$-to-$1$ analytic map from the ball of radius 
$p^{-1/(p-1)}$ centered at the origin to the ball of radius $p^{-1/(p-1)}$ centered at the identity element $\Idop$. Its inverse is given by $\log(\mathbf{a})$. 
\item In these balls, $\exp(a+b)=\exp(a) \exp(b)$ and $\log(ab)=\log(a)+\log(b)$ if $a$ and $b$ commute.
\end{enumerate}
\end{lem}

These basic properties imply that
\begin{equation}
(1+a)^n = \exp(n\log(1+a))
\end{equation}
for any $n\in\N$ and any $a\in A$ with $\|a\|<p^{-1/(p-1)}$. 
In particular, if $\|a\|<p^{-1/(p-1)}$, the map $n\mapsto(1+a)^n$ extends to a Tate analytic map 
\begin{equation}
t\in R\mapsto (1+a)^t :=\exp(t\log(1+a)).
\end{equation}
Since $\|T_f-\Idop\|_{op}\leq p^{-c}$, we can apply this discussion to $a=(T_f-\Idop)\in A$. Doing so, we see that the map $n\mapsto T_f^n$ extends to a Tate analytic map $t\in R\to T_f^t\in A$. Thus, defining
\begin{equation}
\Phi^t = T_f^t(\Id)
\end{equation}
 we obtain a Tate analytic map 
$\Phi\colon R \to R\langle\vx\rangle^m$ that satisfies 
\begin{equation}
\Phi^n=f^n
\end{equation} 
for every $n\in \N$. 
Furthermore, to have $\Phi^t(\vx)\equiv\vx\pmod{p^{c}}$,
we first need to estimate $\exp$ and $\log$:
\begin{lem} \label{lem:estim-des-exp-log}
Let $a\in A$ be an element of a normal $R$-Banach algebra $A$. If $\|a\|<p^{-1/(p-1)}$, then $\|\exp(a)-1\|=\|a\|$ and $\|\log(1+a)\|=\|a\|$. Furthermore, for $t\in R$, we have $\|(1+a)^t-1\|=|t|\cdot\|a\|$.
\end{lem}
\begin{proof}
For the estimate on $\exp$, see \cite[\S{5.7}, Problem 182]{gouvea:p-adic}, which generalizes to normal $R$-Banach algebras. 
The estimate implies that the self-bijection $\mathbf{a}\mapsto\exp(\mathbf{a})-1$ on the ball of radius $p^{-1/(p-1)}$ centered at the origin preserves the norm. Its inverse $\mathbf{a}\mapsto\log(1+\mathbf{a})$ must have the same property. The last claim follows from $(1+a)^t=\exp(t\log(1+a))$.
\end{proof}
Based on this, we have
\begin{align*}
\|\Phi^t-\Id\| &= \|(T^t_f)(\Id)-\Idop(\Id)\| \\
&\leq \|T^t_f-\Idop\|_{op}\|\Id\| \\
&=|t|\cdot\|T_f-\Idop\|_{op} \\
&\leq p^{-c}.
\end{align*}
This proves the theorem of Bell and Poonen. 

\subsection{Formula of the Vector Field} The proof suggested above gives rise to a formula of the vector field. To see how, first note that, on the operator level,
\[\frac1h(T_f^{h}-\Idop)\longrightarrow \log(\Idop+(T_f-\Idop))\]
as $h\to 0$. This follows from the definition of $T_f^h=\exp(h\cdot\log(\Idop+(T_f-\Idop)))$.

Hence if we differentiate $\vt\mapsto\Phi^\vt(\vx)$ at $\vt=0$, we get 
\begin{align}
\Theta_f(\vx) &=\lim_{h\to 0}\frac1h(\Phi^h-\Phi^0)(\vx) \\
&= \lim_{h\to 0}\left(\frac1h(T_f^h-\Idop)\Id\right)(\vx) \\
&= (\log(\Idop+(T_f-\Idop))\Id)(\vx).
\end{align}
This formula can be further understood as
\begin{align}
\Theta_f(\vx) &= (\log(\Idop+(T_f-\Idop))\Id)(\vx) \\
&= \sum_{k=1}^\infty\frac{(-1)^{k-1}}{k}((T_f-\Idop)^k\Id)(\vx) \\
&= \sum_{k=1}^\infty\frac{(-1)^{k-1}}{k}\sum_{j=0}^k\binom{k}{j}(-1)^{k-j}(T_f^j\Id)(\vx) 
\end{align}
which gives 
\begin{align}
\Theta_f(\vx) &= \sum_{k=1}^\infty\sum_{j=0}^k\frac{(-1)^{j-1}}{k}\binom{k}{j}f^j(\vx).\label{eqn:formule-de-champ-vectoriel}
\end{align}

%
%
%
%

\bibliographystyle{plain}
\bibliography{referencesMarkov}

\end{document}